\documentclass[11pt,reqno]{amsart}
\usepackage{graphicx}
\usepackage{amsfonts,amsmath,amssymb}
\newcommand{\cx}{{\mathbb{C}}}

\newcommand {\Q}{\mathcal Q}
\newtheorem{thm}{Theorem}[section]
\newtheorem{propos}[thm]{Proposition}
\newtheorem{corol}[thm]{Corollary}

\theoremstyle{definition}
\newtheorem{dfn}[thm]{Definition}

\newtheorem{rema}[thm]{Remark}
\newtheorem{lem}[thm]{Lemma}
\newcommand{\im}{\ensuremath{\mbox{\rm Im}\,}}
\newcommand{\re}{\ensuremath{\mbox{\rm Re}\,}}

\newcommand{\CC}[1]{\mathbb{C}^{#1}}
\newcommand{\CP}[1]{\mathbb{CP}^{#1}}
\newcommand{\RR}[1]{\mathbb{R}^{#1}}
\newcommand{\dw}{\frac{\partial}{\partial w}}

\newcommand{\dz}{\frac{\partial}{\partial z}}

\newcommand{\lr}{\longrightarrow}

\numberwithin{equation}{section}

\title[Divergent CR-Equivalences]{Divergent CR-Equivalences and Meromorphic Differential Equations}
\author {I. Kossovskiy}
\address{Department of Mathematics, University of Vienna}
\email{ilya.kossovskiy@univie.ac.at}
\author {R. Shafikov}
\address{Department of Mathematics, The University of Western Ontario, London, Ontario N6A 5B7 Canada}
\email{shafikov@uwo.ca}

\begin{document}

\date{\today}

\begin{abstract}
Using the analytic theory of differential equations, we construct
examples of formally but not holomorphically equivalent
real-analytic Levi nonflat hypersurfaces in $\CC{n}$ together with
examples of such hypersurfaces with divergent formal
CR-automorphisms.
\end{abstract}

\maketitle

\tableofcontents

\section{Introduction}

Let $M, M'$ be two smooth real-analytic generic CR-submanifolds in
$\CC{N},\,N\geq 2$, passing through the origin (in what follows we
assume all CR-submanifolds to be generic). The germs $(M,0)$ and
$(M',0)$ of these hypersurfaces at the origin are called \it
holomorphically equivalent, \rm if there exists a germ of an
invertible holomorphic mapping $F:\,(\CC{N},0)\lr(\CC{N},0)$,
called \it a holomorphic equivalence between $(M,0)$ and $(M',0)$,
\rm such that $F(M)\subset M'$. Starting with the celebrated
papers of Poincare\,\cite{poincare}, E.~Cartan\,\cite{cartan},
Tanaka\,\cite{tanaka}, Chern and Moser\,\cite{chern} the
holomorphic equivalence problem for real submanifolds in complex
spaces has been intensively studied. In particular, the following
remarkable fact, demonstrating the difference between complex
analysis in one and several variables, was discovered
in~\cite{chern}. To describe it, we give a few definitions. \it
The type \rm of a CR-submanifold is the pair $(n,k)$, where $n$ is
the CR-dimension and $k$ is the CR-codimension of $M$. A
submanifold of type $(n,k)$ is generic if $N=n+k$. A \it formal
mapping $F:\,(\CC{N},0)\lr(\CC{N},0)$ \rm is an $N$-tuple of
formal power series in $N$ variables without a constant term.  If
$(M,0)$ and $(M',0)$ are the germs at the origin of two smooth
real-analytic CR submanifolds of type $(n,k)$, given by the
defining equations $\theta(z,z)=0$ and $\theta'(z,z)=0$
respectively, we say that $(M,0)$ and $(M',0)$ are \it formally
equivalent, \rm if there exists a formal invertible mapping
$F:\,(\CC{N},0)\lr(\CC{N},0)$, called \it a formal equivalence
between $(M,0)$ and $(M',0)$, \rm and a $k\times k$ matrix-valued
formal power series $\lambda(z,\bar z)$ with an invertible
constant term such that $\theta'\left(F(z),\bar F(\bar
z)\right)=\lambda(z,\bar z)\cdot\theta(z,\bar z)$. Holomorphic
equivalence of hypersurfaces obviously implies that in the formal
category. In the other direction, the convergence of the
Chern-Moser~\cite{chern} normalizing transformation for real
hypersurfaces implies

\smallskip

{\it If two real-analytic hypersurfaces $M,M'\subset \cx^N$ are
Levi nondegenerate, then any formal equivalence between them is in
fact convergent.}

\smallskip

The convergence problem for formal mappings between real
submanifolds is closely related to the problem of analyticity of
smooth CR-mappings (see \cite{ber}). An additional motivation
comes from the fact that, if one has the convergence phenomenon
for some class of CR-submanifolds established, then even a \it
formal \rm normal form  solves the holomorphic equivalence problem
for this class of CR-submanifolds.

Starting with the work of Baouendi, Ebenfelt and Rothschild
\cite{ber0},\cite{ber1} where essential generalizations of the
Chern-Moser convergence phenomenon were obtained, a large number
of papers has been dedicated to the investigation of the
convergence phenomenon. To outline currently known results we
recall that a real-analytic submanifold $M\subset\CC{N}$ is called
\it holomorphically nondegenerate at a point $p\in M$, \rm if it
is not locally holomorphically equivalent near $p$ to a product of
a positive-dimensional complex space and a real submanifold of
smaller dimension (see also \cite{ber} for an alternative
definition). $M$ is called \it minimal at a point $p\in M$, \rm if
there is no germ at $p$ of a proper real submanifold $T\subset M$
of the same CR-dimension (see \cite{tumanov}). Note that
holomorphic nondegeneracy, as well as minimality at a generic
point $q\in M$, are clearly necessary for convergence of all
formal automorphisms of $(M,0)$, and that if $M\subset\CC{N}$ is a
hypersurface, then minimality at $p\in M$ is equivalent to
nonexistence of a complex hypersurface $X\subset M$ passing
through $p$. The study of convergence in the minimal case was
completed in the work of Baouendi, Mir and Rothschild~\cite{bmr},
who proved that formal equivalences between real-analytic minimal
holomorphically nondegenerate CR-submanifolds in complex spaces
always converge. In the context of automorphisms the result was
also obtained by Juhlin and Lamel in \cite{jl1}. In the nonminimal
case, Juhlin and Lamel~\cite{jl2} showed convergence of formal
equivalences for 1-nonminimal hypersurfaces in $\CC{2}$. The main
technical tool employed in the proofs of these positive results is
the finite jet determination phenomenon, discovered in \cite{ber0}
and further developed in subsequent publications (see, e.g.,
\cite{ber1}, \cite{lm},\cite{jl1} and references therein).
Nevertheless, the general question of convergence of formal
equivalences between merely holomorphically nondegenerate
hypersurfaces remained open (see, e.g., \cite{bmr}). In
particular, it was conjectured in \cite{bmr} that the groups
$\mbox{Aut}\,(M,p)$ and $\mathcal F(M,p)$ of respectively
holomorphic and formal self-equivalences of the germ $(M,p)$
 must coincide for a holomorphically nondegenerate hypersurface.

A related question, which remained open for  any type
$(n,k),\,n,k>0$, can be formulated as follows: does the local
holomorphic classification of real-analytic CR-manifolds of type
$(n,k)$ (in particular, real hypersurfaces) coincide with the
formal classification? Other notable results in this direction
were obtained by Moser and Webster \cite{wm}, and by Gong
\cite{gong3}, who found examples of formally but not
holomorphically equivalent real surfaces in $\CC{2}$ near complex
points. We remark that complex points constitute {\it
CR-singularities}, and so such surfaces do not fall into the
category of CR manifolds. Similar results for Lagrangian
submanifolds in $\CC{2}$ are contained in Webster
\cite{websterlagrangian} and Gong \cite{gong2}.

The main results of this paper give the negative answer to both
stated questions and the conjecture in \cite{bmr}. To formulate
the results precisely we need the following definition. Let
$M\subset\CC{2}$ be a real-analytic nonminimal at the origin Levi
nonflat hypersurface (the latter condition means that $M$ is
holomorphically nondegenerate for $M\subset\CC{2}$). Then in
suitable coordinates $(z,w)\in\CC{2}$ near the origin (see, e.g.,
\cite{ebenfelt}) $M$ can be represented by a defining equation
\begin{equation}\label{nonminimality}
\im w = (\re w)^m\Phi(z,\bar z,\re w),
\end{equation}
where the power series $\Phi(z,\bar z,\re w)$  contains no
pluriharmonic terms and also $\Phi(z,\bar z,0)\not\equiv 0$.   The
integer $m\geq 1$ in \eqref{nonminimality}, known to be a
biholomorphic invariant of $(M,0)$, is called \it the
nonminimality order of $M$ at $0$. \rm $M$ given
by~\eqref{nonminimality} is called \it $m$-nonminimal. \rm The
existence of the representation \eqref{nonminimality} is
equivalent to the fact that $M$ is not Levi flat.

\medskip

\noindent \bf Theorem A. \it For any integer $m\geq 2$ there exist
$m$-nonminimal at the origin real-analytic hypersurfaces
$M,M'\subset\CC{2}$ such that the germs $(M,0)$ and $(M',0)$ are
equivalent formally, but are inequivalent holomorphically. \rm

\medskip

The real hypersurfaces in Theorem A can be described explicitly,
namely, using elementary functions and solutions of rational
complex differential equations (see Theorem 4.7 below and also
Remark 5.2). To the best of our knowledge, Theorem A provides the
first known examples of formally but not holomorphically
equivalent CR-manifolds.

The next result shows that the answer is also negative for
automorphisms. Recall that a (formal) holomorphic vector field $L$
near the origin such that its real part $L+\bar L$ is (formally)
tangent to $M$ is called a (formal) infinitesimal automorphism of
$M$.

\medskip

\noindent \bf Theorem B. \it For any integer $m\geq 2$ there
exists an $m$-nonminimal at the origin real-analytic hypersurface
$M\subset\CC{2}$ with a divergent formal infinitesimal
automorphism $L$, vanishing to order $m$ at $0$. In particular,
the real flow $H^t(z,w)$, generated by $L$, consists of divergent
formal automorphisms of the germ $(M,0)$ for any
$t\in\RR{}\setminus \{c\mathbb Z\}$ for some $c\in \RR{}$. \rm

\medskip

In Theorem 5.1 of Section~5 we give a more precise technical
restatement of Theorem B.

It is possible to give a generalization of the phenomenon in
Theorems A and B for higher dimensions. For a real submanifold
$M\subset\CC{N},\,M\ni 0$, we distinguish its stability algebra
$\mathfrak{aut}\,(M,0)$ at the origin and the formal stability
algebra $\mathfrak f\,(M,0)$ (see Section 2 for more details).

\medskip

\noindent \bf Theorem C. \it

\smallskip\noindent (a) For any integers $n,k>0$ there exist real-analytic holomorphically nondegenerate CR-submanifolds
$M,M'\subset\CC{n+k}$ of type $(n,k)$ through the origin such that
the germs $(M,0)$ and $(M',0)$ are equivalent formally but are
inequivalent holomorphically. In particular, the holomorphic and
formal equivalence problems for real-analytic holomorphically
nondegenerate CR-submanifolds of type $(n,k)$ do not coincide.

\smallskip\noindent (b) For any integers $N,m\geq 2$ there exists a real-analytic holomorphically nondegenerate
hypersurface $M\subset\CC{N}$ through the origin with a divergent
formal infinitesimal automorphism $L$, vanishing to order $m$. The
real flow $H^t$ of $L$ consists of divergent formal automorphisms
of the germ $(M,0)$ for any $t\in\RR{}\setminus \{c\mathbb Z\}$
for some $c\in\RR{}$. In particular, the correspondences $\mathcal
H_{N}:\,M\lr\mathfrak{aut}\,(M,0)$ and $\mathcal { F}_{N}:\,
M\lr\mathfrak{f}(M,0)$ between the class of real-analytic
holomorphically nondegenerate real hypersurfaces
$M\subset\CC{N},\,M\ni 0$  and the class of subalgebras in the
algebra $\mathfrak f(\CC{N},0)$ do not coincide. \rm

\medskip

For a real-analytic submanifold $M\subset\CC{N}$ through the
origin one can consider its \it holomorphic isotropy dimension \rm
$\mbox{dim}\,\mathfrak{aut}\,(M,0)$ as well as its \it formal
isotropy dimension \rm $\mbox{dim}\,\mathfrak{f}\,(M,0)$.

\begin{corol} The holomorphic and the formal isotropy dimensions  do not
coincide in general  for a holomorphically nondegenerate
hypersurface $M\subset\CC{N}$.
\end{corol}

The main tool of the paper is the fundamental connection between
CR-geometry and the geometry of completely integrable systems of
complex PDEs, first observed by E.~Cartan and B.~Segre. In
particular, the geometry of real-analytic Levi nondegenerate
hypersurfaces in $\CC{2}$ is closely related to that of second
order ODEs, as discussed in Section~2. For modern treatment of the
subject see also Sukhov~\cite{sukhov1, sukhov2}, Gaussier and
Merker~\cite{gm, merker}, Nurowski and Sparling~\cite{nurowski}.
By discovering a way to connect a certain class of \it nonminimal
\rm real-analytic hypersurfaces in $\CC{2}$ with a class of \it
singular \rm complex linear second order ODEs with a reality
condition, we obtain the desired counterexamples. These examples
arise from certain singular second order ODEs with an isolated
non-Fuchsian (irregular) \rm singularity at the origin.

We point out that the paper contains an intermediate result which
is the characterization of nonminimal at the origin and spherical
at a generic point real hypersurfaces having the infinitesimal
automorphism $iz\dz$ (``rotations inside the complex tangent"),
see Theorem~3.15 and the algorithm at the end of Section~3.
Real-analytic hypersurfaces of this type were intensively studied
in the work of Ebenfelt, Lamel and Zaitsev \cite{elz}, Beloshapka
\cite{belnew}, Kolar and Lamel \cite{kl} and the authors in
\cite{nonminimal}. As the construction of each single example in
the cited papers is technically quite involved, the explicit
description, given in Section 3 of this paper, is of independent
interest. In fact, one can show that this description is \it
complete \rm (see Remark 3.19).

\smallskip

The paper is organized as follows. Because we use tools from a
broad range of topics in complex analysis and dynamical systems,
we provide in Section 2 relevant background material. In Section 3
we introduce a class of 2-parameter families of planar holomorphic
curves, that can be potentially the Segre families of nonminimal
at the origin and \it spherical at a generic point \rm real
hypersurfaces, and, at the same time, serve as a family of
integral curves for certain second order linear ODEs with an
isolated meromorphic singularity (we call these {\it
$m$-admissible ODEs with a real structure}). The explicit
characterization of these ODEs, given in Theorem 3.15, allows us
to obtain in Section~4 nonminimal real hypersurfaces, for which
the associated ODE has, essentially, the prescribed behaviour of
solutions. Then, by finding a divergent formal equivalence between
holomorphically inequivalent ODEs with a real structure, we obtain
in Propositions 4.2 and 4.3 the potential formal equivalence, and
the rest of the section is dedicated to proving that this formal
mapping is the mapping between the initial real hypersurfaces,
which proves Theorem~A and the first statement of Theorem~C. In
Section 5 we apply the divergent transformation from Theorem~A to
the infinitesimal automorphisms, which gives the proof of
Theorem~B and the second statement in Theorem~C. We also give a
description of the hypersurface $M'$ from Theorem~A by elementary
functions, and a hint of similar description for $M$ (see Remark
5.2). Finally, we formulate some open problems and conjectures,
arising from the results of this paper.

\smallskip

\begin{center}\bf Acknowledgments \end{center}

\smallskip

We would like to thank I.~Vyugin from the Steklov Institute,
Moscow and S.~Yakovenko from Weizmann Institute of Science,
Rehovot for useful discussions.

The first author is supported by the Austrian Science Fund (FWF),
and the second author is partially supported by the Natural
Sciences and Engineering Research Council of Canada.

\section{Preliminaries and background material}

\subsection{Segre varieties.}
Let $M$ be a smooth connected real-analytic hypersurface in
$\cx^{n+k}$ of type $(n,k)$, $n,k>0$, $0\in M$, and $U$ a
neighbourhood of the origin where $M\cap U$ admits a real-analytic
defining function $\phi(Z,\overline Z)$. For every point $\zeta\in
U$ we can associate to $M$ its so-called Segre variety in $U$
defined as
$$
Q_\zeta= \{Z\in U : \phi(Z,\overline \zeta)=0\}.
$$
Segre varieties depend holomorphically on the variable $\overline
\zeta$. One can find  a suitable pair of neighbourhoods $U_2={\
U_2^z}\times U_2^w\subset \cx^{n}\times \CC{k}$ and $U_1 \Subset
U_2$ such that
$$
  Q_\zeta=\left \{(z,w)\in U^z_2 \times U^w_2: w = h(z,\overline \zeta)\right\}, \ \ \zeta\in U_1,
$$
is a closed complex analytic graph. Here $h$ is a holomorphic
function. Following \cite{DiPi} we call $U_1, U_2$ a {\it standard
pair of neighbourhoods} of the origin. The antiholomorphic
$(n+k)$-parameter family of complex hypersurfaces
$\{Q_\zeta\}_{\zeta\in U_1}$ is called \it the Segre family of $M$
at the origin. \rm From the definition and the reality condition
on the defining function the following basic properties of Segre
varieties follow:
\begin{equation}\label{e.svp}
  Z\in Q_\zeta \ \Leftrightarrow \ \zeta\in Q_Z,
\end{equation}
\begin{equation*}
  Z\in Q_Z \ \Leftrightarrow \ Z\in M,
\end{equation*}
\begin{equation*}
  \zeta\in M \Leftrightarrow \{Z\in U_1: Q_\zeta=Q_Z\}\subset M.
\end{equation*}
The fundamental role of Segre varieties for holomorphic mappings
is illuminated by their invariance property: if $f: U \to U'$ is a
holomorphic map sending a smooth real-analytic submanifold
$M\subset U$ into another such submanifold $M'\subset U'$, and $U$
is as above, then
$$
f(Z)=Z' \ \ \Longrightarrow \ \ f(Q_Z)\subset Q'_{Z'}.
$$ For
the proofs of these and other properties of Segre varieties see,
e.g., \cite{webster}, \cite{DiFo}, \cite{DiPi}, \cite{shaf}, or
\cite{ber}.

In the particularly important case when $M$ is a \it real
hyperquadric, \rm i.e., when $$M=\left\{
[\zeta_0,\dots,\zeta_N]\in \cx\mathbb P^{N} : H(\zeta,\bar \zeta)
=0  \right\},$$ where $H(\zeta,\bar \zeta)$ is a nondegenerate
Hermitian form in $\CC{N+1}$ with $k+1$ positive and $l+1$
negative eigenvalues, $k+l=N-1,\,0\leq l\leq k\leq N-1$, the Segre
variety of a point $\zeta \in\CP{N}$ is the projective hyperplane
$ Q_\zeta = \{\xi\in \cx\mathbb P^N: H(\xi,\bar\zeta)=0\}$. The
Segre family $\{Q_\zeta,\,\zeta\in\CP{N}\}$  coincides in this
case with the space $(\CP{N})^*$ of all projective hyperplanes in
$\CP{N}$.

The space of Segre varieties $\{Q_Z : Z\in U_1\}$ can be
identified with a subset of $\cx^K$ for some $K>0$ in such a way
that the so-called \it Segre map \rm $\lambda : Z \to Q_Z$ is
holomorphic (see \cite{DiFo}). For a Levi nondegenerate at a point
$p$ hypersurface $M$ its Segre map is one-to-one in a
neighbourhood of $p$. When $M$ contains a complex hypersurface
$X$, for any point $p\in X$ we have $Q_p = X$ and $Q_p\cap
X\neq\emptyset\Leftrightarrow p\in X$, so that the Segre map
$\lambda$ sends the entire $X$ to a unique point in $\CC{N}$ and,
accordingly, $\lambda$ is not even finite-to-one near each $p\in
X$ (i.e., $M$ is \it not essentially finite \rm at points $p\in
X$). For a hyperquadric $\Q\subset\CP{N}$ the Segre map $\lambda'$
is a global natural one-to-one correspondence between $\CP{N}$ and
the space $(\CP{N})^*$.

\subsection{Real hypersurfaces and second order differential equations.}
Using the Segre family of a Levi nondegenerate real hypersurface
$M\subset\CC{N}$ , one can associate to it a system of
second-order holomorphic PDEs with $1$ dependent and $N-1$
independent variables. The corresponding remarkable construction
goes back to E.~Cartan \cite{cartanODE},\cite{cartan} and Segre
\cite{segre}, and was recently revisited in
\cite{sukhov1},\cite{sukhov2},\cite{nurowski},\cite{gm},\cite{merker}
(see also references therein). We describe here the procedure for
the case $N=2$, which will be relevant for our purposes. In what
follows we denote the coordinates in $\CC{2}$ by $(z,w)$, and put
$z=x+iy,\,w=u+iv$. Let $M\subset\CC{2}$ be a smooth real-analytic
hypersurface, passing through the origin, and let $(U_1,U_2)$ be
its standard pair of neighbourhoods. In this case
 one associates to $M$ a second-order holomorphic ODE, uniquely determined by the condition that it is satisfied by the
Segre family $\{Q_\zeta\}_{\zeta\in U_1}$ of $M$ in a
neighbourhood of the origin where the Segre varieties are
considered as graphs $w=w(z)$. More precisely, it follows from the
Levi nondegeneracy of $M$ near the origin that the Segre map
$\zeta\lr Q_\zeta$ is injective and also that the Segre family has
the so-called transversality property: if two distinct Segre
varieties intersect at a point $q\in U_2$, then their intersection
at $q$ is transverse. Thus, $\{Q_\zeta\}_{\zeta\in U_1}$ is a
2-parameter holomorphic w.r.t. $\bar\zeta$ family of holomorphic
curves in $U_2$ with the transversality property. It follows from
the holomorphic version of the fundamental ODE theorem (see, e.g.,
\cite{ilyashenko}) that there exists a unique second order
holomorphic ODE $w''=\Phi(z,w,w')$, satisfied by the graphs
$\{Q_\zeta\}_{\zeta\in U_1}$.

This procedure can be made more explicit if one considers the
so-called \it complex defining
 equation \rm (see, e.g., \cite{ber})\,
$w=\rho(z,\bar z,\bar w)$ \, of $M$ near the origin, which one
obtains by substituting $u=\frac{1}{2}(w+\bar
w),\,v=\frac{1}{2i}(w-\bar w)$ into the real defining equation and
applying the holomorphic implicit function theorem. The complex
defining function $\rho$ here satisfies an additional reality
condition
\begin{equation}\label{realty}
w\equiv\rho(z,\bar z,\bar\rho(\bar z,z,w)),
\end{equation}
reflecting the fact that $M$ is a real hypersurface. The Segre
variety $Q_p$ of a point $p=(a,b)$, close to the origin, is given
by
\begin{equation} \label{segre}w=\rho(z,\bar a,\bar b). \end{equation}
Differentiating \eqref{segre} once, we obtain
\begin{equation}\label{segreder} w'=\rho_z(z,\bar a,\bar b). \end{equation}
Considering \eqref{segre} and \eqref{segreder}  as a holomorphic
system of equations with the unknowns $\bar a,\bar b$, and
applying the implicit function theorem near the origin, we get
$$
\bar a=A(z,w,w'),\,\bar b=B(z,w,w').
$$
The implicit function theorem here is applicable because the
Jacobian of the system coincides with the Levi determinant of $M$
for $(z,w)\in M$ (see, e.g., \cite{merker}). Differentiating
\eqref{segre} twice and plugging there the expressions for $\bar
a,\bar b$ finally yields
\begin{equation}\label{segreder2}
w''=\rho_{zz}(z,A(z,w,w'),B(z,w,w'))=:\Phi(z,w,w').
\end{equation}
Now \eqref{segreder2} is the desired holomorphic second-order ODE
$\mathcal E$.

The concept of a PDE system associated with a CR-manifold can be
generalized for the class of arbitrary $l$-nondegenerate, $l\geq
1$, CR-submanifolds (see \cite{ber} for the definition of this
nondegeneracy condition). Namely, for any $l$-nondegenerate
CR-submanifold $M\subset\CC{n+k}$ of the CR-dimension $n$ and the
codimension $k$ one can assign a completely integrable system
$\mathcal E(M)$ of holomorphic PDEs with $n$ independent and $k$
dependent variables. The correspondence $M\lr \mathcal E(M)$ has
the following fundamental properties:

\begin{enumerate}

\item[(1)] Every local holomorphic equivalence $F:\, (M,0)\lr (M',0)$
between two $l$-nondegenerate CR-submanifolds is an equivalence
between the corresponding PDE systems $\mathcal E(M),\mathcal
E(M')$;

\item[(2)] The complexification of the infinitesimal automorphism algebra
$\mathfrak{hol}(M,0)$ of $M$ at the origin coincides with the Lie
symmetry algebra  of the associated PDE system $\mathcal E(M)$
(see, e.g., \cite{olver} for the details of the concept).

\end{enumerate}

For the proof and applications of the properties (1) and (2) we
refer to
\cite{sukhov1},\cite{sukhov2},\cite{nurowski},\cite{gm},\cite{merker}.
We emphasize that for a nonminimal at the origin hypersurface
$M\subset\CC{2}$ there is \it no \rm a priori way to associate to
$M$ a second-order ODE or even a more general PDE system near the
origin. However, in Section 3 we provide a way to connect a
special class of nonminimal real hypersurfaces in $\CC{2}$ with a
class of complex linear differential equations with an isolated
singularity.

\subsection{Equivalences of differential equations} For simplicity we consider here only
scalar ordinary differential equations, even though all the
constructions below can be applied for arbitrary systems of PDEs.
We refer to the book of Olver~\cite{olver} as a general reference
to this subsection. Also note that these constructions are nothing
but a simple interpretation of a more general concept of a \it jet
bundle \rm.

Consider two ODEs $\mathcal
E=\{y^{(n)}=\Phi(x,y,y',...,y^{(n-1)})\}$ and $\mathcal
E^*=\{y^{(n)}=\Phi^*(x,y,y',...,y^{(n-1)})\}$, where the functions
$\Phi$ and $\Phi^*$ are holomorphic in some neighbourhood the
origin in $\CC{n}$. We say that a biholomorphism
$F:\,(\CC{n+1},0)\lr(\CC{n+1},0)$  transforms $\mathcal E$ into
$\mathcal E^*$, if it sends (locally) graphs of solutions of
$\mathcal E$ into graphs of solutions of $\mathcal E^*$.
Introducing the $(n+2)$-dimensional \it $n$-jet space \rm
$J^{(n)}$, which is a linear space with the coordinates
$x,y,y_1,...,y_{n}$, corresponding to the independent variable
$x$, the dependent variable $y$ and its derivatives up to order
$n$, one can naturally consider $\mathcal E$ and $\mathcal E^*$ as
complex submanifolds in $J^{(n)}$. Moreover, for any
biholomorphism $F$ as above, sufficiently close to the origin, one
may consider its \it $n$-jet prolongation
$F^{(n)}:\,(J^{(n)},0)\lr (J^{(n)},0)$. \rm The jet prolongation
procedure can be conveniently interpreted as follows. The first
two components of the mapping $F^{(n)}$ coincide with those of
$F$. To obtain the remaining components  we  denote the
coordinates in the preimage by $(x,y)$ and in the target domain by
$(X,Y)$. Then the derivative $\frac{dY}{dX}$ can be symbolically
recalculated, using the chain rule, in terms of $x,y,y'$, so that
the third coordinate $Y_1$ in the target jet space becomes a
function of $x,y,y_1$. In the same manner one obtains all the $n$
missing components of the prolongation of the mapping $F$. It is
then nothing but a tautology to say that \it the mapping $F$
transforms the ODE $\mathcal E$ into $\mathcal E^*$ if and only if
the prolonged mapping $F^{(n)}$ transforms $(\mathcal E,0)$ into
$(\mathcal E^*,0)$ as submanifolds in the jet space $J^{(n)}$\rm.
A similar statement can be formulated for certain singular
differential equations, for example, for linear ODEs (see, e.g.,
\cite{ilyashenko}).

For $n=2$ the local equivalence problem for nonsingular ODEs was
solved in the celebrated papers of E.~Cartan \cite{cartanODE} and
A.~Tresse \cite{tresse}. Of particular interest to us is the
special case when the ODE is equivalent to the simplest (flat)
equation $y''=0$. We refer to the book of Arnold \cite{arnoldgeom}
for a modern treatment of the problem and some further
developments.

\subsection{Formal power series, formal equivalences and formal flows}
For the set-up and basic properties of  formal power series and
formal mappings we refer to \cite{ilyashenko} and \cite{ber}. We
give below a list of statements that will be useful for us in what
follows.

\smallskip

$\bullet$ The substitution of a formal mapping
$(\CC{n},0)\lr(\CC{n},0)$ into a formal power series is
well-defined. In particular, a composition of two formal mappings
$(\CC{n},0)\lr(\CC{n},0)$ is always well-defined (recall that for
a formal mapping $(\CC{n},0)\lr(\CC{m},0)$ we always assume the
absence of the constant term).

\smallskip

$\bullet$ A formal mapping $F:(\CC{n},0)\lr(\CC{n},0)$ is called
\it invertible \rm if there exists a formal mapping
$G:\,(\CC{n},0)\lr(\CC{n},0)$ with $F\circ G$ being the identity
map. Note that any formal mapping $(\CC{n},0)\lr(\CC{n},0)$  is
formally invertible, provided its linear part is invertible as an
element of $\mbox{GL}_n(\CC{})$.

\smallskip

$\bullet$ For any formal mapping
$F(z,w):\,(\CC{m}\times\CC{n},0)\lr(\CC{n},0)$ the following
formal version of the implicit function theorem holds: if the
linear part $\frac{\partial F}{\partial w}(0)$ of $F$ w.r.t. $w$
is invertible, then there exists a unique formal mapping
$\varphi:\,(\CC{m},0)\lr(\CC{n},0)$  such that $F(z,\varphi(z))=0$
as a formal vector power series.

\smallskip

Let $X=f_1(z)\frac{\partial}{\partial
z_1}+...+f_n(z)\frac{\partial}{\partial z_n}$ be a formal vector
field with $X(0)=0$. \it A formal flow of $X$ \rm is a holomorphic
w.r.t. $t\in\CC{}$ one-parameter family of formal mappings
$F^t(z):\,(\CC{n},0)\lr(\CC{n},0)$ such that
$\frac{d}{dt}F^t(z)|_{t=0}=X$ and the mapping $t\lr F^t$ is a
group homomorphism between $(\CC{},+)$ and the group of formal
invertible mappings $(\CC{n},0)\lr(\CC{n},0)$. A 1-parameter group
$F^t(z)$ as above is called \it holomorphic, \rm if all the
truncations $j^k F^t(z)$ are holomorphic w.r.t. $t$.

\smallskip

$\bullet$ For any formal vector field $X$ with $X(0)=0$ its formal
flow always exists and can be uniquely determined.

\smallskip

Recall that for a real submanifold $M\subset\CC{N}$ with $M\ni 0$
its \it infinitesimal automorphism algebra at the origin \rm is
the real Lie algebra $\mathfrak{hol}\,(M,0)$ of holomorphic vector
fields $X$ near the origin such that their real parts $X+\bar X$
are tangent to $M$ at each point. The \it stability algebra \rm
$\mathfrak{aut}\,(M,0)\subset\mathfrak{hol}\,(M,0)$ is the
subalgebra of vector fields, vanishing at $0$. Infinitesimal
automorphisms are exactly the vector fields with a flow,
generating local automorphisms $F:\,(M,0)\lr(M,0)$. One can
further define the \it formal infinitesimal automorphism algebra
\rm $\mathfrak f\,(M,0)$, \rm which consists of formal vector
fields in $\CC{N}$, formally satisfying the tangency condition to
$M$, and the \it formal stability algebra \rm  $\mathfrak
f\,(M,0)$, which consists of formal vectors fields $X\in f\,(M,0)$
with $X(0)=0$.

\smallskip

$\bullet$ A formal vector field $X$ with $X(0)=0$ is a formal
infinitesimal automorphism of $(M,0)$ if and only if the formal
flow of $X$ formally preserves the germ $(M,0)$.

\smallskip

Finally, we will need the following property of formal
CR-mappings. For a real-analytic submanifold $M\subset\CC{N}$,
passing through the origin and given in some neighbourhood $U\ni
0$ by the defining equation $\theta(z,\bar z)=0$, we define its
\it complexification \rm to be the complex submanifold
$$M^{\CC{}}=\{(z,\zeta)\in U\times
U:\,\theta(z,\zeta)=0\}\subset\CC{2N}.$$

\smallskip

$\bullet$ Let $M_1,M_2\subset\CC{N}$ be real-analytic
submanifolds, passing through the origin. A (formal)
transformation $F:\,(\CC{N},0)\lr(\CC{N},0)$ without a free term
sends (formally) $(M_1,0)$ into $(M_2,0)$ if and only if the
direct product
 $\left(F(z),\bar
F(\zeta)\right):\,(\CC{2N},0)\lr(\CC{2N},0)$ (called \it the
complexification of $F$) \rm sends (formally) $(M_1^{\CC{}},0)$
into $(M_2^{\CC{}},0)$.

\subsection{Complex linear differential equations with an isolated singularity}
Perhaps the most important and geometrical class of complex
differential equations is the class of complex linear ODEs. We
refer to \cite{ilyashenko}, \cite{ai}, \cite{bolibruh},
\cite{vazow},\cite{coddington} and references therein for various
facts and problems, concerning complex linear differential
equations. A first order linear system of $n$ complex ODEs in a
domain $G\subset\CC{}$ (or simply a linear system in a domain $G$
in what follows) is a holomorphic ODE system $\mathcal L$ of the
form $y'(w)=A(w)y$, where $A(w)$ is an $n\times n$ matrix-valued
holomorphic in $G$
 function and $y(w)=(y_1(w),...,y_n(w))$
is an $n$-tuple of unknown functions. Solutions of $\mathcal L$
near a point $p\in G$ form a linear space of dimension~$n$.
Moreover, all the solution $y(w)$ of $\mathcal L$ are defined
globally in $G$ as (possibly multiple-valued) analytic functions,
i.e., any germ of a solution near a point $p\in G$ of $\mathcal L$
extends analytically along any path $\gamma\subset G$, starting
at~$p$. A \it fundamental system of solutions for $\mathcal L$ \rm
is a matrix whose columns form some collection of $n$ linearly
independent solutions of $\mathcal L$.

If the case when $G$ is a punctured disc, centred at $0$, we call
$\mathcal L$ \it a system with an isolated singularity at $w=0$.
\rm An important (and sometimes even a complete) characterization
of an isolated singularity is its \it monodromy operator, \rm
defined as follows. If $Y(w)$ is some fundamental system of
solutions of $\mathcal L$ in $G$, and $\gamma$ is a simple loop
about the origin, then it is not difficult to see that the
monodromy of $Y(w)$ w.r.t. $\gamma$ is given by the right
multiplication by a constant nondegenerate matrix $M$, called \it
the monodromy matrix. \rm The matrix $M$ is defined up to a
similarity, so that it defines a linear operator
$\CC{n}\lr\CC{n}$, which is  called the monodromy operator of the
singularity.

If the matrix-valued function $A(w)$ is meromorphic at the
singularity $w=0$, we call it a {\it meromorphic singularity}. As
the solutions of $\mathcal L$ are holomorphic in any proper sector
$S\subset G$ of a sufficiently small radius with the vertex at
$w=0$, it is important to study the behaviour of the solutions as
$w\rightarrow 0$. If all solutions of $\mathcal L$ admit a bound
$||y(w)||\leq C|w|^A$  in any such sector (with some constants
$C>0,\ A\in \mathbb R$, depending possibly on the sector), then
$w=0$ is called \it a regular singularity, \rm otherwise it is
called \it an irregular singularity. \rm In particular, in the
case of the trivial monodromy the singularity is regular if and
only if all the solutions of $\mathcal L$ are meromorphic in $G$.
L.~Fuchs introduced the following condition: a singular point
$w=0$ is called \it Fuchsian, \rm if $A(w)$ is meromorphic  at
$w=0$ and has a pole of order $\leq 1$ there. The Fuchsian
condition turns out to be sufficient for the regularity of a
singular point. Another remarkable property of Fuchsian
singularities can be described as follows. We call two complex
linear systems with an isolated singularity $\mathcal L_1,\mathcal
L_2$ \it (formally) equivalent, \rm if there exists a (formal)
transformation $F:\,(\CC{n+1},0)\lr(\CC{n+1},0)$ of the form
$F(w,y)=(w,H(w)y)$ for some (formally) invertible matrix-valued
function $H(w)$, which transforms (formally) $\mathcal L_1$ into
$\mathcal L_2$. It turns out that two Fuchsian systems are
formally equivalent if and only if they are equivalent
holomorphically (moreover, any formal equivalence between them as
above must be convergent). However, this is \it not \rm the case
for non-Fuchsian systems (see \cite{vazow} for some related
constructions).

A scalar linear complex ODE of order $n$ in a
 domain $G\subset\CC{}$ is an ODE $\mathcal E$ of the form
$$z^{(n)}=a_{n}(w)z+a_{n-1}(w)z'+...+a_1(w)z^{(n-1)},$$ where $\{a_j(w)\}_{j=1,...,n}$ is
a given collection of holomorphic functions in $G$ and $z(w)$ is
the unknown function. By a reduction of $\mathcal E$ to a first
order linear system (see the above references and
also~\cite{vyugin} for various approaches of doing that) one can
naturally transfer most of the definitions and facts, relevant to
linear systems, to scalar equations of order $n$. The main
difference here is contained in the appropriate definition of
Fuchsian: a singular point $w=0$ for an ODE $\mathcal E$ is called
\it Fuchsian, \rm if the orders of poles $p_j$ of the functions
$a_j(w)$ satisfy the inequalities $p_j\leq j$, $j=1,2,\dots,n$. It
turns out that the condition of Fuchs becomes also necessary for
the regularity of a singular point in the case of $n$-th order
scalar ODEs.

Further information on the classification of isolated
singularities (including Poincare-Dulac normalization) can be
found in \cite{ilyashenko}, \cite{vazow} or \cite{coddington}.

\section{Meromorphic linear differential equations with real
structure}

The main purpose of this section is to establish a class of
 complex linear second order ODEs with a meromorphic
 singularity, that generate, in a certain sense, nonminimal at the origin and
 spherical at a generic point real hypersurfaces. We start with a  number of definitions.
Denote by $\Delta_\varepsilon$ a disc in $\CC{}$, centred at $w=0$
of radius $\varepsilon$, and by $\Delta^*_\varepsilon$ the
corresponding punctured disc.

 \begin{dfn}
 A complex linear second order ODE with an isolated singularity at the origin is called
 \it $m$-admissible, \rm if it is of the form
 \begin{equation}\label{admissibleODE}
 z''=\frac{P(w)}{w^m}z'+\frac{Q(w)}{w^{2m}}z,
 \end{equation}
 where $m\geq 1$ is an integer and $P(w),Q(w)\in\mathcal O(\Delta_\varepsilon)$ for some
 $\varepsilon>0$.
 \end{dfn}

Direct calculations show that if a germ $z(w)$ of a solution of
\eqref{admissibleODE} is
 invertible in some domain, then the inverse function $w(z)$
 satisfies in the image domain the ODE \begin{equation}
 \label{inverse}w''=-\frac{P(w)}{w^m}(w')^2-\frac{Q(w)}{w^{2m}}(w')^3\,z.\end{equation}
 We call \eqref{inverse} \it the inverse ODE \rm for \eqref{admissibleODE}.
Note that in \eqref{admissibleODE} the independent variable is
$w$, while $z$ is the independent variable for the inverse ODE.
Also note that without the requirement that $\frac{P(w)}{w^m}$ and
$\frac{Q(w)}{w^{2m}}$ are irreducible, a meromorphic at the origin
complex linear ODE is admissible for different integers $m\geq 1$.

 We next introduce a class of anti-holomorphic 2-parameter families of planar
 complex  curves that potentially can be the family of solutions for an
 $m$-admissible ODE and, at the same time, the family of Segre varieties of a real hypersurface in $\CC{2}$.

 \begin{dfn} \it An $m$-admissible Segre family \rm is a
 2-parameter antiholomorphic family of planar holomorphic curves
 in a polydisc $\Delta_\delta\times\Delta_\varepsilon$
 of the form
 \begin{equation} \label{admissiblefamily} w=\bar\eta
 e^{\pm i\bar\eta^{m-1}\psi(z\bar\xi,\bar\eta)},\end{equation} where
 $m\geq 1$ is an integer, $\xi\in\Delta_\delta,\eta\in \Delta_\varepsilon$ are holomorphic parameters, and the function $\psi(x,y)$ is holomorphic
 in the polydisc $\Delta_{\delta^2}\times\Delta_\varepsilon$ and has there an expansion
 $$
 \psi(x,y)=x+\sum\limits_{k\geq 2}\psi_k(y)x^k, \ \ \psi_k(y)\in\mathcal O(\Delta_\varepsilon) .
 $$
 \end{dfn}

 It follows then that an $m$-admissible Segre family has the form
 \begin{equation}\label{phi}
 \mathcal S = \left\{w=\bar\eta e^{\pm i\bar\eta^{m-1}\left(z\bar\xi+\sum_{k\geq
 2}\psi_k(\bar\eta)z^k\bar\xi^k\right)},\ (\xi,\eta)\in \Delta_{\delta}\times\Delta_\varepsilon \right\}.
 \end{equation}

The fact that an anti-holomorphic 2-parameter family of planar
complex curves is $m$-admissible can be easily checked: a family
 $w=\rho(z\bar\xi,\bar\eta)$, where $\rho$ is holomorphic in some polydisc $U\subset\CC{2}$, centred
 at the origin, is $m$-admissible if and only if the defining function
 $\rho$ has the expansion
 $\rho(z\bar\xi,\bar\eta)=\bar\eta\pm i\bar\eta^m
 z\bar\xi+O(\bar\eta^m z^2\bar\xi^2)$.

 For any real-analytic nonminimal at the origin hypersurface $M\subset\CC{2}$ with nonminimality
 order~$m$ of the form
 \begin{equation}\label{admissiblehyper}
 v=u^m\left(\pm |z|^2+\sum\limits_{k\geq 2}h_{k}(u)|z|^{2k}\right),
 \end{equation}
 it is not difficult to check that its Segre family is an $m$-admissible Segre family. We call a real hypersurface of
 the form \eqref{admissiblehyper} an \it  $m$-admissible nonminimal hypersurface. \rm Note that in the case of
 $m$-admissible Segre families (respectively, nonminimal hypersurfaces) the integer $m$ is uniquely
 determined by the Segre family (respectively, by the hypersurface). Depending on the sign in the exponent
 $e^{\pm i\bar\eta^{m-1}\psi(z\bar\xi,\bar\eta)}$ we call an $m$-admissible Segre family
 \it positive \rm or \it negative \rm respectively, and the same for real hypersurfaces. In analogy with the case of real hypersurfaces, we call the holomorphic curve in the family
 \eqref{admissiblefamily},  corresponding to the values
 $\xi=a,\eta=b$ of parameters, \it the Segre variety of a point
 \rm  $p=(a,b)\in\Delta_\delta\times\Delta_\varepsilon$ and denote it by $Q_p$.
We call the hypersurface
 $$
 X=\{w=0\}\subset\Delta_\delta\times\Delta_\varepsilon
 $$
\it the singular locus \rm of an $m$-admissible Segre family. The
following proposition provides some simple properties of Segre
families.

 \begin{propos} The following properties hold for an
 $m$-admissible family:

 \smallskip

 (i) $Q_p\cap X\neq\emptyset \Longleftrightarrow p\in X \Longleftrightarrow Q_p=X$.

 \smallskip

 (ii) The Segre mapping $\lambda:\,p\longrightarrow Q_p$ is
 injective in $(\Delta_\delta\times\Delta_\varepsilon)\setminus X$.


\end{propos}

\begin{proof}
The first property follows directly from~\eqref{admissiblefamily}.
For the proof of (ii) we note that if a Segre variety $Q_p$ is
given as a graph $w=w(z)$, then, from \eqref{admissiblefamily},
$w(0)=\bar\eta,\,w'(0)=\pm i\bar\xi\bar\eta^m$, depending on the
sign of the Segre family, and that implies the global injectivity
of $\lambda(p)$ in
$(\Delta_\delta\times\Delta_\varepsilon)\setminus X$.
\end{proof}

The next definition connects admissible Segre families with second
order linear ODEs with a meromorphic singularity.

 \begin{dfn}
 We say that an $m$-admissible Segre family $\mathcal S$  \it is associated with an $m$-admissible ODE $\mathcal
 E$, \rm if after an appropriate shrinking of the basic neighbourhood $\Delta_\delta\times\Delta_\varepsilon$ of the
 origin all the elements $Q_p\in\mathcal S$ with $p\notin X$, considered as graphs
 $w=w(z)$, satisfy the inverse ODE for $\mathcal E$.
 \end{dfn}

 Given an ODE $\mathcal E$, we denote an associated $m$-admissible Segre family by
 $\mathcal S^\pm_{\mathcal E}$, depending on the sign of the Segre family.
By Proposition 3.3, $w\ne 0$ for $p\notin X$, and so we may alway
substitute the Segre
 varieties into \eqref{inverse}.

 \begin{propos} For any $m$-admissible ODE $\mathcal E$~\eqref{admissibleODE} there
 exists a unique positive and a unique negative $m$-admissible (with the same $m$) Segre family
 $\mathcal S$, associated with $\mathcal E$. The ODE $\mathcal E$ and the associated
 Segre families $\mathcal S^\pm_{\mathcal E}$ given by \eqref{phi}, satisfy the following relations:
 \begin{gather}\label{relation22}
 P(w)=\pm 2i\psi_2(w)-w^{m-1} ,\\
 \label{relation33} Q(w)=6\psi_3(w)-8(\psi_2(w))^2\pm 2i(m-1)w^{m-1}\psi_2(w)\mp 2iw^m\psi_2'(w).
 \end{gather}
 In particular, for any fixed $m$ the correspondences $\mathcal E\longrightarrow\mathcal S^+_{\mathcal E}$
 and $\mathcal E\longrightarrow\mathcal S^-_{\mathcal E}$ are injective.
 \end{propos}

 \begin{proof}
Consider a positive $m$-admissible Segre family $\mathcal S$, as
in \eqref{admissiblefamily}, and an $m$-admissible ODE $\mathcal
E$. We first express the condition that $\mathcal S$ is associated
with $\mathcal E$ in the form of a differential equation. Fix
$p=(\xi,\eta)\in\Delta_\delta\times\Delta_\varepsilon$ and
consider the Segre variety $Q_p$, given by
\eqref{admissiblefamily} as a graph $w=w(z)$. For the function
$\psi(x,y)$ we denote by $\dot\psi$ and $\ddot\psi$ its first and
second derivatives respectively w.r.t. the first argument. Then
one computes
\begin{eqnarray*}
w' &=& i\bar\xi\bar\eta^m e^{i\bar\eta^{m-1}\psi(z\bar\xi,\bar\eta)}\dot\psi(z\bar\xi,\bar\eta),\\
w'' &=& i\bar\xi^2\bar\eta^m
e^{i\bar\eta^{m-1}\psi(z\bar\xi,\bar\eta)}\ddot\psi(z\bar\xi,\bar\eta)-
\bar\xi^2\bar\eta^{2m-1}e^{i\bar\eta^{m-1}\psi(z\bar\xi,\bar\eta)}(\dot\psi(z\bar\xi,\bar\eta))^2.
\end{eqnarray*}
Plugging these expressions into \eqref{inverse} yields after
simplifications
\begin{gather}\label{findphi}
\ddot\psi(z\bar\xi,\bar\eta)=-i(\dot\psi(z\bar\xi,\bar\eta))^2\left(\bar\eta^{m-1}+
P(\bar\eta e^{i\bar\eta^{m-1}\psi(z\bar\xi,\bar\eta)})e^{i(1-m)\bar\eta^{m-1}\psi(z\bar\xi,\bar\eta)}\right)+\\
\notag +(\dot\psi(z\bar\xi,\bar\eta))^3Q(\bar\eta
e^{i\bar\eta^{m-1}\psi(z\bar\xi,\bar\eta)})e^{i(2-2m)\bar\eta^{m-1}\psi(z\bar\xi,\bar\eta)}z\bar\xi
.
\end{gather}
Consider now a holomorphic near the origin differential equation
\begin{equation}\label{findphi2}
y''=-i(y')^2\left(\bar\eta^{m-1}+ P(\bar\eta
e^{i\bar\eta^{m-1}y})e^{i(1-m)\bar\eta^{m-1}y}\right)
+(y')^3tQ(\bar\eta
e^{i\bar\eta^{m-1}y})e^{i(2-2m)\bar\eta^{m-1}y},
\end{equation}
where $y$ is the dependent variable, $t$ is the independent
variable, and $\bar\eta$ is a holomorphic parameter near the
origin. The Cauchy problem for the ODE \eqref{findphi2} with the
initial data $y(0)=0,\,y'(0)=1$ is well-posed, as the right-hand
side is polynomial w.r.t. $y'$. As follows from the theorem on the
analytic dependence of solutions of a holomorphic ODE on a
holomorphic parameter (see, e.g., \cite{ilyashenko}), its solution
$y=y_0(t,\bar\eta)$ is unique and holomorphic in some polydisc
$U\subset\CC{2}$, centred at the origin. The comparison of
\eqref{findphi} and \eqref{findphi2} shows that the functions
$y_0(z\bar\xi,\bar\eta)$ and $\psi(z\bar\xi,\bar\eta)$ coincide.
Observe that the above arguments are reversible.

For the proof of the proposition, given an $m$-admissible ODE
$\mathcal E$, we solve the corresponding equation \eqref{findphi2}
with the initial data $y(0)=0,\,y'(0)=1$, and obtain a solution
$y_0(t,\bar\eta)=t+\sum\limits_{k\geq 2,l\geq
0}c_{kl}t^k\bar\eta^l$. Then $$w=\bar\eta
e^{i\bar\eta^{m-1}y_0(z\bar\xi,\bar\eta)}$$ is the desired
positive
 $m$-admissible Segre family $\mathcal S=\mathcal
S_{\mathcal E}$ associated with $\mathcal E$. The uniqueness of
$\mathcal S_{\mathcal E}$ also follows from the uniqueness of the
solution of the Cauchy problem.

To prove the relations \eqref{relation22},\eqref{relation33}, we
substitute \eqref{admissiblefamily} into \eqref{inverse}. As
$(\bar\xi,\bar\eta)\in\Delta_\delta\times\Delta_\varepsilon$ is
arbitrary, we compare in the obtained identity the
$z^0\bar\xi^2\bar\eta^l$-terms, which gives
$2i\bar\eta^m\psi_2(\bar\eta)-\bar\eta^{2m-1}=\bar\eta^mP(\bar\eta)$.
This is equivalent to \eqref{relation22}. Comparing then the
$z^1\bar\xi^3\bar\eta^l$-terms, we get
\begin{gather*}6i\bar\eta^m\psi_3(\bar\eta)-6\bar\eta^{2m-1}\psi_2(\bar\eta)-i\bar\eta^{3m-2}=i\bar\eta^mQ(\bar\eta)
-2iP(\bar\eta)(2i\bar\eta^m\psi_2(\bar\eta)-\bar\eta^{2m-1})-\\-i
m P(\bar\eta)\bar\eta^{2m-1}+i\bar\eta^{2m}P'(\bar\eta) .
\end{gather*}
From this and \eqref{relation22}, we finally obtain
\eqref{relation33}.

The proof for a negative Segre family is analogous. \end{proof}

 Proposition 3.5 gives an effective algorithm for computing the
 $m$-admissible Segre family for a given linear meromorphic
 second order ODE. Our goal is, however, to identify those
 ODEs that produce Segre families with a reality
 condition, that is, Segre families of nonminimal real
 hypersurfaces.

 \begin{dfn} We say that an $m$-admissible Segre family \it has a
 real structure, \rm if it is the Segre family of an
 $m$-admissible real hypersurface $M\subset \CC{2}$. We also say
 that an $m$-admissible ODE $\mathcal E$ \it has a positive (respectively, negative) real structure,
 \rm if the associated positive (respectively, negative) $m$-admissible Segre family $\mathcal S^\pm_{\mathcal E}$
 has a real structure. We say that the corresponding real hypersurface $M$ is \it associated with $\mathcal E$. \rm
 \end{dfn}

 We will need a development of the following construction from the theory of second order ODEs, going back to A.Tresse
 \cite{tresse}  and  E.Cartan \cite{cartanODE} (see also \cite{arnoldgeom}, \cite{nurowski}, \cite{merker}, \cite{gm} and
 references therein). Let $\rho(z,\bar\xi,\bar\eta)$ be a holomorphic function near the origin in $\CC{3}$ with
  $\rho(0,0,0)=0$, and $d\rho(0,0,0)=\bar\eta$. For  $z,\xi\in\Delta_\delta,w,\eta\in\Delta_\varepsilon$, let
 $$
 \mathcal S=\{w=\rho(z,\bar \xi,\bar \eta)\}
 $$
be a 2-parameter antiholomorphic family of holomorphic curves near
the origin, parametrized by $(\xi,\eta)$. We will call such a
family \it a (general) Segre family, \rm and for each point
 $p=(\xi,\eta)\in\Delta_\delta\times\Delta_\varepsilon$ we call
 the corresponding holomorphic curve
 $Q_p=\{w=\rho(z,\bar \xi,\bar \eta)\}\in \mathcal S$
its \it  Segre variety. \rm Clearly, an $m$-admissible Segre
family is a particular example of a general Segre family.

We say that two (general) Segre families \it coincide, \rm if
there exists a nonempty open neighbourhood $G$ of the origin such
that for any point $p\in G$ the Segre varieties of $p$ in both
families coincide. Further, given a (general) Segre family
$\mathcal S$, from the implicit function theorem one concludes
that the antiholomorphic family of planar holomorphic curves
$$\mathcal
 S^*=\{\bar\eta=\rho(\bar\xi,z,w)\}$$ is also a (general) Segre family for
 some, possibly, smaller polydisc $\Delta_{\tilde\delta} \times \Delta_{\tilde\varepsilon}$.

 \begin{dfn}
 The Segre family $\mathcal S^*$ is called the \it dual Segre family for $\mathcal S$. \rm
 \end{dfn}

 The dual Segre family has a simple interpretation: in the defining equation of the family $S$ one should consider the
 parameters $\bar\xi,\bar\eta$ as new coordinates, and the variables $z,w$ as new parameters.
 We denote the Segre variety of a point $p$ with respect to the family $\mathcal S^*$ by $Q^*_p$.

\begin{lem} Suppose that $\mathcal S$ is a
positive (respectively, negative) $m$-admissible Segre family.
Then $\mathcal S^*$ is a negative (respectively, positive)
$m$-admissible Segre family. \end{lem}

\begin{proof} To obtain the defining function $\rho^*(z,\bar\xi,\bar\eta)$ of the general Segre
family $\mathcal S^*$ we solve for $w$ its defining equation
\begin{equation}\label{dual}
\bar\eta=w e^{\pm iw^{m-1}\left(z\bar\xi+\sum_{k\geq
 2}\psi_k(w)z^k\bar\xi^k\right)} .
\end{equation}
Note that \eqref{dual} implies
\begin{equation}\label{dual2}
w=\bar\eta e^{\mp iw^{m-1}(z\bar\xi+O(z^2\bar\xi^2))}.
\end{equation}
We then obtain from \eqref{dual2}
$w=\rho^*(z,\bar\xi,\bar\eta)=\bar\eta(1+O(z\bar\xi))$.
Substituting the latter representation into \eqref{dual2} gives
$w=\rho^*(z,\bar\xi,\bar\eta)=\bar\eta e^{\mp
i\bar\eta^{m-1}(z\bar\xi+O(z^2\bar\xi^2))}$, which proves the
lemma.
\end{proof}

We also consider the following Segre family, connected with $S$:
$$\bar{\mathcal S}=\{w=\bar\rho(z,\bar \xi,\bar \eta)\}.$$

 \begin{dfn} The Segre family $\bar{\mathcal S}$ is called \it the
 conjugated family of $\mathcal S$. \rm
 \end{dfn}

 If $\sigma: \CC{2}\longrightarrow\CC{2}$ is the antiholomorphic involution $(z,w)\longrightarrow(\bar
 z,\bar w)$, then one simply has $\sigma(Q_p) = \overline {Q_{\sigma(p)}}$.
 We  will denote the Segre variety of a point $p$ with respect to the family $\bar{\mathcal S}$ by
 $\bar Q_p$. It follows from the definition that if $\mathcal S$ is a
positive (respectively, negative) $m$-admissible Segre family,
then $\bar{\mathcal S}$ is a negative (respectively, positive)
$m$-admissible Segre family.

 In the same manner as for the case of an $m$-admissible Segre
 family, we say that a (general) Segre family $\mathcal S=\{w=\rho(z,\bar \xi,\bar
 \eta)\}$ has a \it real structure, \rm if there exists a smooth real-analytic
 hypersurface $M\subset\CC{2}$, passing through the origin, such
 that $\mathcal S$ is the Segre family of $M$.

 The use of the  dual and the conjugated Segre families stems from the following

 \begin{propos} A (general) Segre family $\mathcal S$ has a real
 structure if and only if the dual Segre family $S^*$ coincides
 with the conjugated one: $\mathcal S^*=\bar{\mathcal
 S}$.\end{propos}

 \begin{proof} Suppose that $\mathcal S$ is the Segre family at the origin of a
 real hypersurface $M\subset\CC{2}$ with the complex defining equation $w=\rho(z,\bar z,\bar
 w)$. Then $\mathcal S$ is given by $\{w=\rho(z,\bar \xi,\bar
 \eta)\},$  and if $(z,w)\in
Q^*_{(\xi,\eta)}$, then $\bar\eta=\rho(\bar\xi,z,w)$, so that
$(\bar \xi,\bar \eta)\in Q_{(\bar z,\bar w)}$. Then~\eqref{e.svp}
gives $(\bar z,\bar w)\in Q_{(\bar \xi,\bar\eta)}$, and so $(z,w)
\in \sigma(Q_{(\bar\xi,\bar\eta)}) = \bar Q_{(\xi,\eta)}$. In the
same way one shows that $(z,w)\in\bar Q_{(\xi,\eta)}$ implies
$(z,w)\in Q^*_{(\xi,\eta)}$, so that $\mathcal S^*=\bar{\mathcal
 S}$.

 If it is given now that $\mathcal S^*=\bar{\mathcal
 S}$, then
 $\left[\bar\eta=\rho(\bar\xi,z,w)\right]\Longleftrightarrow\left[w=\bar\rho(z,\bar \xi,\bar
 \eta)\right]$, which is possible only if
 $$\bar\eta\equiv\rho(\bar\xi,z,\bar\rho(z,\bar\xi,\bar\eta)).$$
 Changing notations and replacing in the latter identity the variables $\bar\eta,\bar\xi,z$
 by the variables  $w,z,\bar\xi$ respectively, we obtain the
 complexification of the reality condition \eqref{realty}. Hence, the equation $w=\rho(z,\bar z,\bar w)$
 determines the germ at the origin of a smooth real-analytic hypersurface $M$. This proves the proposition.
\end{proof}

We next transfer the above real structure criterion from
$m$-admissible families to the associated ODEs.

\begin{dfn}
Let $\mathcal E$ be an $m$-admissible ODE. We say that an
$m$-admissible  ODE $\mathcal E^*$ is \it dual  to $\mathcal E$,
\rm if the negative $m$-admissible Segre family dual to the family
$\mathcal S_{\mathcal E}^+$ is associated with $\mathcal E^*$,
i.e.,
$$
\mathcal E^* {\rm \ is\ dual\ to\ } \mathcal E \
\Longleftrightarrow \ (\mathcal S^+_{\mathcal E})^* =
S^-_{\mathcal E^*}.
$$
In the same manner, we say that an $m$-admissible ODE
$\overline{\mathcal E}$ is \it conjugated to $\mathcal E$, \rm if
the negative $m$-admissible Segre family conjugated to the family
$\mathcal S_{\mathcal E}^+$, is associated with
$\overline{\mathcal E}$, i.e.,
$$
\bar{\mathcal E} {\rm \ is\ conjugated\ to\ } \mathcal E \
\Longleftrightarrow \ \bar{\mathcal S}^+_{\mathcal E} =
S^-_{\bar{\mathcal E}}.
$$
\end{dfn}

From Proposition 3.5 we conclude that both the conjugated and the
dual ODEs are unique (if exist). The existence of the conjugated
ODE is obvious: if $\mathcal E$ is given by
$z''=\frac{P(w)}{w^m}z'+\frac{Q(w)}{w^{2m}}z,$ then, clearly, the
desired ODE $\overline{\mathcal E}$  is given explicitly by
\begin{equation}\label{conjugatedODE}
z''=\frac{\bar P(w)}{w^m}z'+\frac{\bar
Q(w)}{w^{2m}}z.\end{equation} However, the existence of the dual
ODE is a more subtle issue. To prove it, we first need

\begin{propos}[Transversality Lemma]
Let $\mathcal S$ be an $m$-admissible Segre family in a polydisc
$\Delta_\delta\times\Delta_\varepsilon$, and $X$ be its singular
locus. After possibly shrinking the polydisc
$\Delta_\delta\times\Delta_\varepsilon$,
 the following property holds: if $p,q\in(\Delta_\delta\times\Delta_\varepsilon)\setminus X$,
 $p\ne q$, and $Q_p$ and $Q_q$ intersect at a point $r$, then their
intersection at $r$ is transverse.
\end{propos}

\begin{proof}
Suppose first that $\mathcal S$ is positive. Take an arbitrary
$p=(\xi,\eta)\in (\Delta_\delta\times\Delta_\varepsilon)\setminus
X$ and consider $Q_p$ as a graph $w=w(z)=\bar\eta
e^{i\bar\eta^{m-1}(z\bar\xi+O(z^2\bar\xi^2))}.$ Then
\begin{equation}\label{approximation}
w=\bar\eta+O(z\bar\xi\bar\eta),\,\frac{w'}{w^m}=i\bar\xi+O(z\bar\xi).
\end{equation}
The latter implies that by shrinking the polydisc
$\Delta_\delta\times\Delta_\varepsilon$, one can make the map
$$
(\xi,\eta)\longrightarrow\left(w(z),\frac{w'(z)}{w^m(z)}\right),$$
which is defined for each $z$, injective in
$(\Delta_\delta\times\Delta_\varepsilon)\setminus X$ (once for all
$z$). Then the same property holds for the map
$$
(\xi,\eta)\longrightarrow\left(w(z),w'(z)\right),
$$
which shows that the graphs $Q_p$ and $Q_q$ cannot have the same
slope at a point of intersection. The proof for the negative case
is analogous. \end{proof}

\begin{propos} Let $\mathcal E$ be an $m$-admissible ODE.
Then the dual ODE $\mathcal E^*$ always exists.\end{propos}

\begin{proof}
Let $\Delta_\delta\times\Delta_\varepsilon$ be the polydisc where
$\mathcal S^+_{\mathcal E}$ is defined, and $X$ be the singular
locus. For simplicity, we will assume that the dual family is
defined in the same polydisc. Consider two (possibly
multiple-valued) linearly independent solutions $h_1(w),h_2(w)$ of
$\mathcal E$ in the punctured disc $\Delta^*_\varepsilon$. Then,
by the definition of the associated Segre family, for any
$p=(\xi,\eta)\in (\Delta_\delta\times\Delta_\varepsilon)\setminus
X$, $\xi\neq 0$, $\eta\ne 0$, there exist unique complex numbers
$\lambda_1(\bar\xi,\bar\eta),\lambda_2(\bar\xi,\bar\eta)$ such
that $Q_p$ is contained in the graph
$$z=\lambda_1(\bar\xi,\bar\eta)h_1(w)+\lambda_2(\bar\xi,\bar\eta)h_2(w).$$
As the family $S$ depends on the parameters holomorphically,
$\lambda_1(\bar\xi,\bar\eta),\lambda_2(\bar\xi,\bar\eta)$ are two
(possibly multiple-valued) analytic functions in
$\Delta^*_\delta\times\Delta^*_\varepsilon$.

We claim that $\bar\xi\lambda_1(\bar\xi,\bar\eta)$ and
$\bar\xi\lambda_2(\bar\xi,\bar\eta)$ are independent of $\bar
\xi$. Indeed, we note that from the defining equation~\eqref{phi}
the expression $z\bar\xi$ for the family $\mathcal S$ depends only
on $w$ and $\bar\eta$, so that in some polydisc $U$ in
$\Delta_\delta\times\Delta^*_\epsilon\times\Delta_\delta\times\Delta^*_\epsilon$
we have $\lambda_1(\bar\xi,\bar\eta)\bar\xi
h_1(w)+\lambda_2(\bar\xi,\bar\eta)\bar\xi h_2(w)=\Psi(w,\bar\eta)$
for an appropriate holomorphic function $\Psi$. Differentiating
the latter equality w.r.t. $w$ and solving a system of linear
equations w.r.t. $\lambda_1\bar\xi,\lambda_2\bar\xi$, we get
\begin{equation}\label{separate}
(\lambda_1(\bar\xi,\bar\eta)\bar\xi,\lambda_2(\bar\xi,\bar\eta)\bar\xi)=(\Psi(w,\bar\eta),\Psi_w(w,\bar\eta))\cdot
H^{-1}(w),
\end{equation}
where $H(w)$ is the Wronski matrix for the linearly independent
functions $h_1(w),h_2(w)$ (we consider the single-valued branches
of these functions, defined in the polydisc $U$). As the
right-hand side of~\eqref{separate} depends on $\bar\eta$ only, we
conclude that
$\lambda_1(\bar\xi,\bar\eta)\bar\xi,\lambda_2(\bar\xi,\bar\eta)\bar\xi$
are independent of $\bar\xi$, which proves the claim.

It follows from the claim that each $Q_p$ as above is contained in
the graph
$$z\bar\xi=\tau_1(\bar\eta)h_1(w)+\tau_2(\bar\eta)h_2(w)$$ for
some (possibly multiple-valued) analytic in $\Delta^*_\varepsilon$
functions $\tau_1(\bar\eta),\tau_2(\bar\eta)$. It follows then
that for any $p=(\xi,\eta)\in
(\Delta_\delta\times\Delta_\varepsilon)\setminus X,\,\xi\neq 0$,
the dual Segre variety $Q^*_p$  is contained in the graph
\begin{equation}\label{dualgraph}
z\bar\xi=\tau_1(w)h_1(\bar\eta)+\tau_2(w)h_2(\bar\eta).\end{equation}

We claim now that the Wronskian $d(w)=\begin{vmatrix} \tau_1(w) &
\tau_2(w)\\ \tau'_1(w) & \tau'_2(w)\end{vmatrix}$ does not vanish
 in $\Delta^*_\varepsilon$. Indeed, suppose otherwise, that $d(w_0)=0$
 for some $w_0$, and let $(0,w_0) \in Q^*_{(\xi_0,\eta_0)}$ for some $(\xi_0,\eta_0),\,\xi_0\neq 0$ (one can take $(\xi_0,\eta_0)=(\xi,\bar w_0)$ for some $\xi\in \Delta^*_\delta$).
 We seek all $(\xi,\eta)$ such that $Q^*_{(\xi,\eta)}$ passes through  $(0,w_0)$ and has
 $1$-jet there the same as $Q^*_{(\xi_0,\eta_0)}$. Clearly, such $(\xi,\eta)$ are given by
 \begin{equation}\label{e.h12}
\left( \begin{matrix} \tau_1(w_0) & \tau_2(w_0)\\ \tau'_1(w_0) &
\tau'_2(w_0)
\end{matrix} \right)\cdot \frac{1}{\bar\xi}
\left( \begin{matrix} h_1(\bar\eta) \\
h_2 (\bar\eta)
\end{matrix} \right) = \left( \begin{matrix} 0 \\ \alpha
\end{matrix} \right) ,
\end{equation}
 where $\alpha = (h_1(\bar\eta_0)\tau_1'(w_0) + h_2(\bar\eta_0)\tau_2'(w_0))/\bar\xi_0$.
If we think of $\left(\frac{1}{\overline\xi}h_1,
\frac{1}{\overline\xi}h_2\right)$ as the unknown variables in the
above linear system, then since $d(w_0)=0$, its solution contains
an affine line $L$, passing through
$\left(\frac{1}{\overline\xi_0}h_1(\bar\eta_0),
\frac{1}{\overline\xi_0}h_2(\bar\eta_0)\right)$. The linear
independence of $h_1(w)$ and $h_2(w)$ implies that the map
$H:\,(\xi,\eta)\lr\left(\frac{1}{\xi}h_1(\eta),
\frac{1}{\xi}h_2(\eta)\right)$ is locally biholomorphic near
$(\bar\xi_0,\bar\eta_0)$ so that there exist points $(\xi,\eta)$
near $(\xi_0,\eta_0)$ with $(\xi,\eta)\neq (\xi_0,\eta_0)$ and
$H(\bar\xi,\bar\eta)\in L$. We conclude that there exists a
$1$-dimensional family of dual Segre varieties $Q^*_p$, passing
through the point $(0,w_0)$ that have the same 1-jet at the point
$w=w_0$.  But this contradicts Proposition~3.12, and so $d(w_0)\ne
0$.

It follows that the graphs \eqref{dualgraph} satisfy the linear
differential equation
$$W(z,\tau_1,\tau_2)=\begin{vmatrix} z & \tau_1(w) & \tau_2(w)\\ z' & \tau'_1(w) &
\tau'_2(w)\\ z'' & \tau''_1(w) & \tau''_2(w)\end{vmatrix}=0,$$
which can be rewritten as
\begin{equation}\label{dualdirect}
z''=A(w)z'+B(w)z,
\end{equation}
with its inverse ODE equal to
\begin{equation}\label{dualinverse}
w''=-A(w)(w')^2-B(w)(w')^3 z
\end{equation}
for some holomorphic in $\Delta^*_\varepsilon$ functions
$A(w),B(w)$. The relation \eqref{dualinverse} is satisfied by
every dual Segre variety $Q^*_p$, $p\in
(\Delta_\delta\times\Delta_\varepsilon)\setminus X$.

On the other hand, we may consider relations
\eqref{approximation}, applied to the dual family $\mathcal S^*$,
and use them to obtain a second order ODE satisfied by all
$Q^*_p,p\in (\Delta_\delta\times\Delta_\varepsilon)\setminus X$.
To do that, we apply the implicit function theorem in
\eqref{approximation} and obtain
$$\bar\xi=\Lambda\left(z,w,\frac{w'}{w^m}\right),\,\bar\eta=\Omega\left(z,w,\frac{w'}{w^m}\right)$$
for some functions
$\Lambda(z,w,\zeta)=i\zeta+O(z\zeta),\Omega(z,w,\zeta)=w+O(zw\zeta)$,
holomorphic in a polydisc $V\subset\CC{3}$, centred at the origin.
We next differentiate twice the relation \eqref{phi}, applied to
the dual family $\mathcal S^*$, w.r.t. $z$ and get
$w''=O(\bar\xi^2\bar\eta^m)$. Plugging in the latter
representation
$\bar\xi=\Lambda(z,w,\frac{w'}{w^m}),\,\bar\eta=\Omega(z,w,\frac{w'}{w^m})$,
one gets a second order ODE \begin{equation}\label{dualODE}
w''=\Phi\left(z,w,\frac{w'}{w^m}\right)\end{equation} for some
function $\Phi(z,w,\zeta)$, holomorphic in a polydisc $\tilde
V\subset\CC{3}$, centred at the origin (compare this with the
procedure in Section~2.2). The ODE \eqref{dualODE} is satisfied by
all $Q^*_p$ with $p\in
(\Delta_\delta\times\Delta_\varepsilon)\setminus X$. The function
$\Phi(z,w,\zeta)$ also satisfies $\Phi(z,w,\zeta)=O(\zeta^2w^m)$.
We now compare \eqref{dualODE} with \eqref{dualinverse}. We put
$\zeta:=\frac{w'}{w^m}$ and observe that in some domain $G\subset
\tilde V,$ $\Phi(z,w,\zeta)=-A(w)w^{2m}\zeta^2-B(w)w^{3m}\zeta^3
z$, which shows that the function $\Phi(z,w,\zeta)$ is cubic
w.r.t. the third argument. Since, in addition,
$\Phi(z,w,\zeta)=O(\zeta^2w^m)$, we conclude  that the function
$\Phi(z,w,\zeta)$ has the form
$w^m(\Phi_2(w)\zeta^2+\Phi_3(w)\zeta^3 z)$ for some functions
$\Phi_2(w)$ and $\Phi_3(w)$ holomorphic in a disc
$\Delta_r\subset\CC{},r>0$. Then the substitution
$\zeta=\frac{w'}{w^m}$ turns \eqref{dualODE} into an
$m$-admissible ODE, rewritten in the inverse form. This proves the
proposition.
\end{proof}

Combining Proposition 3.13 with Propositions 3.5 and 3.10, we
immediately obtain a crucial

\begin{corol} An $m$-admissible ODE $\mathcal E$ has a positive real
structure if and only if the conjugated ODE coincides with the
dual one: $\mathcal E^*=\overline{\mathcal E}$. \end{corol}

It is possible now to prove the main result of this section.

\begin{thm}
Let $\mathcal E:\,z''=\frac{P(w)}{w^m}z'+\frac{Q(w)}{w^{2m}}z$ be
an $m$-admissible ODE, $w\in\Delta_r,r>0$. Then $\mathcal E$ has a
positive real structure if and only if the functions $P(w),Q(w)$
have the form:
\begin{equation}\label{keyformulas}
P(w)=2ia(w)-mw^{m-1},\ Q(w)=b(w)+iw^ma'(w),\end{equation} where
$a(w)=\sum\limits_{j=0}^\infty a_jw^j,\,a_j\in\RR{}$, and
$b(w)=\sum\limits_{j=0}^\infty b_jw^j,\,b_j\in\RR{}$, are
convergent in $\Delta_r$ power series. Moreover, if $\mathcal E$
has a positive real structure, then the associated real
hypersurface $M\subset\CC{2}$ is Levi nondegenerate and spherical
outside the complex locus $X=\{w=0\}$.
\end{thm}

\begin{proof} Let $\mathcal E^*$ be given as
$z''=\frac{P^*(w)}{w^m}z'+\frac{Q^*(w)}{w^{2m}}z$. As previously
observed, the conjugated  ODE $\overline{\mathcal E}$ has the form
$z''=\frac{\bar P(w)}{w^m}z'+\frac{\bar Q(w)}{w^{2m}}z$.
 Let $\mathcal S=\mathcal S^+_{\mathcal E}$
be given in a polydisc $\Delta_\delta\times\Delta_\varepsilon$ by
$w=\bar\eta
 e^{i\bar\eta^{m-1}\psi(z\bar\xi,\bar\eta)}$ with $\psi$ as in
 \eqref{phi} and $\mathcal S^*$ be given (in the same polydisc, for simplicity) by $w=\bar\eta
 e^{-i\bar\eta^{m-1}\psi^*(z\bar\xi,\bar\eta)}$ with $\psi^*$ as
 in \eqref{phi}. Then $\overline{\mathcal S}$ is given by $w=\bar\eta
 e^{-i\bar\eta^{m-1}\bar\psi(z\bar\xi,\bar\eta)}$. According to
 Corollary 3.14, $\mathcal E$ has a real structure if an only if
 $\bar P(w)=P^*(w)$ and $\bar Q(w)=Q^*(w)$. It follows from
 \eqref{relation22},\eqref{relation33} that the latter conditions are equivalent to
 \begin{equation}\label{upto3jet}
 \bar\psi_2=\psi^*_2,\ \ \bar\psi_3=\psi^*_3,
 \end{equation}
 so that one has to develop condition \eqref{upto3jet}. By
 the definition of the dual family, one has $$[\bar\eta=w
 e^{iw^{m-1}\psi(z\bar\xi,w)}]\,\Longleftrightarrow\,[w=\bar\eta
 e^{- i\bar\eta^{m-1}\psi^*(z\bar\xi,\bar\eta)}],$$ and, using the expansion \eqref{phi}, it is not
 difficult to obtain from here that
 \begin{gather} \label{dualfunction}
 z\bar\xi+\psi_2(w)z^2\bar\xi^2+\psi_3(w)z^3\bar\xi^3+O(z^4\bar\xi^4)=\left(z\bar\xi+\psi^*_2(\bar\eta)z^2\bar\xi^2+\psi^*_3(w)z^3\bar\xi^3+O(z^4\bar\xi^4)\right)\times\\
 \notag\times e^{i(m-1)w^{m-1}(
 z\bar\xi+\psi_2(w)z^2\bar\xi^2+O(z^3\bar\xi^3))}\left|_{\bar\eta=w+iw^mz\bar\xi+O(z^2\bar\xi^2)}\right.\end{gather}
 Gathering in \eqref{dualfunction} terms with $z^2\bar\xi^2$ and $z^3\bar\xi^3$ respectively, one
 gets
 $$\psi_2=\psi_2^*+i(m-1)w^{m-1},\,\psi_3=\psi_3^*+iw^m(\psi_2^*)'+i(m-1)w^{m-1}\psi_2-\frac{1}{2}(m-1)^2w^{2m-2}+i(m-1)w^{m-1}\psi_2^*.$$
 In view of the two latter identities, one can verify that
 \eqref{upto3jet} can be rewritten as
 \begin{equation}\label{keyformulas1}
 \psi_2(w)=\lambda(w)+i\frac{m-1}{2}w^{m-1},\,\psi_3(w)=\mu(w)+
 \frac{i}{2}w^m\lambda'(w)+i(m-1)w^{m-1}\lambda(w),
 \end{equation}
 where $\lambda(w),\mu(w)$ are two convergent in $\Delta_r$ power
 series with \it real \rm coefficients. Applying
 \eqref{relation22},\eqref{relation33} again, we conclude that
 \eqref{keyformulas1} is equivalent to \begin{equation}\label{keyformulas2}
 P(w)=2i\lambda(w)-mw^{m-1},\,Q(w)=6\mu(w)-8\lambda^2(w)+iw^{m}\lambda'(w)+2(m-1)^2w^{2m-2},\end{equation}
 which is already equivalent to \eqref{keyformulas} after setting
 \begin{equation}\label{ablambdamu}
a(w):=\lambda(w),\,b(w):=6\mu(w)-8\lambda^2(w)+2(m-1)^2w^{2m-2}.\end{equation}

 It remains to prove that if $\mathcal E$ has a real structure,
 then the associated nonminimal real hypersurface $M\subset\CC{2}$
 is Levi nondegenerate and spherical in $M\setminus X$, where $X$
 is the singular locus of the Segre family $\mathcal S$ (and, at the same time, the nonminimal locus of
 $M$). Fix a point $p\in M\setminus X$ and its small neighbourhood $V$. It follows from Proposition 3.12
 that if two  Segre varieties of $M$ intersect at a point $r\in V$, then their intersection is
 transverse. Accordingly, any Segre variety of $M$ near the point $p$ is determined by its
 1-jet at a given point uniquely. The latter fact implies (see, e.g.,
 \cite{DiPi},\cite{ber}) that $M$ is Levi nondegenerate at $p$.
 Finally, to prove that $M$ is spherical at $p=(z_0,w_0),w_0\neq 0,$ we argue as in the
 proof of Proposition 3.13: fix two linearly independent
 solutions $h_1(w),h_2(w)$ of $\mathcal E$ in $\Delta^*_\varepsilon$. Then each $Q_q$ with $q=(\xi,\eta)\notin X$
 is contained in the graph
\begin{equation}\label{graph}z\bar\xi=\tau_1(\bar\eta)h_1(w)+\tau_2(\bar\eta)h_2(w)\end{equation} for
some (possibly multiple-valued) analytic in $\Delta^*_\varepsilon$
functions $\tau_1(\bar\eta),\tau_2(\bar\eta)$. We then use
slightly modified arguments from \cite{arnoldgeom} to construct
the desired mapping into a sphere:
since the  Wronskian $d(w)=\begin{vmatrix} h_1(w) & h_2(w)\\
h'_1(w) & h'_2(w)\end{vmatrix}$ is non-zero in
$\Delta^*_\varepsilon$, we may suppose that either $h_1(w_0)\neq
0$ or $h_2(w_0)\neq 0$ (for some fixed analytic  elements of
$h_1,h_2$ in $V$). If, for example, $h_1(w_0)\neq 0$, consider in
 $V$  the mapping
\begin{equation}\label{maptosphere}
\Lambda:\,(z,w)\longrightarrow\left(\frac{z}{h_1(w)},\frac{h_2(w)}{h_1(w)}\right).
\end{equation}
As the Wronskian $d(w)$ is non-zero in $V$, we may assume that
$\Lambda$ is biholomorphic there. By the definition of $\Lambda$,
the graphs \eqref{graph} are the preimages of complex lines under
the map $\Lambda$, so that $\Lambda$ maps Segre varieties of $M$
into complex lines. It is not difficult to verify from here that
$\Lambda(M)$ is contained in a quadric $\mathcal Q\subset\CP{2}$
(see, for example, the proof of Theorem 6.1 in \cite{nonminimal}),
which implies sphericity of $M$ at $p$. The theorem is completely
proved now. \end{proof}

\begin{rema} It is possible to give also the
characterization of the ODEs with a negative real structure: these
are obtained by conjugating ODEs with a positive real structure.
\end{rema}

\begin{rema}
It follows from \eqref{keyformulas} that a complex linear ODE with
an isolated meromorphic singularity at the origin is
$m$-admissible with a positive real structure for at most one
value $m\in\mathbb{Z}_+$.
\end{rema}

\begin{rema} Theorem 3.15, combined with the proof of Proposition 3.5, gives an effective algorithm for
obtaining nonminimal at the origin real hypersurfaces
$M\subset\CC{2}$ with prescribed nonminimality order $m\geq 1$,
Levi nondegenerate and spherical outside the nonminimal locus
$X\subset M$, and invariant under the group $z\longrightarrow
e^{it}z$ of rotational symmetries. Moreover, one can prescribe
essentially arbitrary 6-jet to the hypersurface $M$. For reader's
convenience we summarize this algorithm below.

\medskip

\begin{center} \bf Algorithm for obtaining nonminimal spherical real
hypersurfaces \end{center}

\medskip

1. Take arbitrary convergent in some disc centred at the origin
power series $a(w),b(w)$ with real coefficients, and compute two
functions $P(w),Q(w)$ by the formulas \eqref{keyformulas}. This
gives an $m$-admissible ODE \eqref{admissibleODE}.

\smallskip

2. Solve the holomorphic ODE \eqref{findphi2} with a holomorphic
parameter $\bar\eta$ and the initial data $y(0)=0,\,y'(0)=1$. This
gives a holomorphic near the origin in $\CC{2}$ function
$\psi(t,\bar\eta)$.

\smallskip

3. Then the equation $w=\bar w e^{i\bar w^{m-1}\psi(z\bar z,\bar
w)}$ determines an invariant under the group of rotational
symmetries nonminimal at the origin real hypersurface
$M\subset\CC{2}$ of nonminimality order $m$, Levi nondegenerate
and spherical outside the nonminimal locus $X=\{w=0\}$. The 6-jet
of $M$ is determined by finding $\lambda(w),\mu(w)$ using
\eqref{ablambdamu} and then $\psi_2,\psi_3$ by formulas
\eqref{keyformulas1}.
\end{rema}

\begin{rema} Theorem 3.15 and the algorithm above provide in fact
a \it complete \rm description of nonminimal at the origin
real-analytic Levi nonflat hypersurfaces $M\subset\CC{2}$, Levi
nondegenerate and spherical outside the complex locus, such that
$iz\dz\in\mathfrak{aut}\,(M,0)$. In order to prove that one needs
to associate to each $M$ as above a second order $m$-admissible
ODE. The fact that every nonminimal spherical $M$ admits an ODE
associated with it is proved in our upcoming paper \cite{KS-new}.
\end{rema}

\section{Formally but not holomorphically equivalent real
hypersurfaces}

In this section we will use the explicit description of linear
meromorphic ODEs with a real structure given by Theorem~3.15 to
construct for each fixed nonminimality order $m\geq 2$ a family of
pairwise formally equivalent $m$-nonminimal at the origin real
hypersurfaces, Levi nondegenerate and spherical outside the
nonminimal locus, which are, however, generically pairwise
holomorphically inequivalent at the origin. The construction is
based on existence of families of linear ODEs with a meromorphic
singularity at the origin and a positive real structure, with the
property that the ODEs in the family are pairwise formally but not
holomorphically equivalent.

The desired ODEs and the associated real hypersurfaces are
introduced as follows. Fix an integer $m\geq 2$ and put
$a(w)\equiv 1$ and $b(w)=\beta w^{2m-2}$, where $\beta\in\RR{}$ is
a real constant. Applying now formulas \eqref{keyformulas}, we
obtain the following one-parameter family $\mathcal E^m_\beta$ of
complex linear ODEs with a meromorphic singularity at the origin,
which are $m$-admissible and have a positive real structure:

\begin{equation}\label{mainODE}
z''=\left(\frac{2i}{w^m}-\frac{m}{w}\right)z'+\frac{\beta}{w^2}z.
\end{equation}
As $m\geq 2$, each $\mathcal E^m_\beta$ has a \it non-Fuchsian \rm
singularity at the origin, which plays a crucial role in our
construction. We denote by $M^m_\beta$ the $m$-nonminimal at the
origin real hypersurfaces, associated with $\mathcal E^m_\beta$.
Each $M^m_\beta$ is Levi nondegenerate and spherical outside the
complex locus $X=\{w=0\}$.

Introducing a new dependent variable $u:=z'w$, one can rewrite
\eqref{mainODE} as a first order system
\begin{equation} \label{mainsystem} \begin{pmatrix} z \\ u
\end{pmatrix}'=\left[\frac{1}{w^m}\begin{pmatrix} 0 & 0 \\ 0 & 2i
\end{pmatrix}+\frac{1}{w}\begin{pmatrix} 0 & 1 \\ \beta & 1-m
\end{pmatrix}\right]\begin{pmatrix} z \\ u
\end{pmatrix}\end{equation} with a non-Fuchsian singularity at the
origin.

\begin{dfn} A \it (formal) gauge transformation \rm is a
(formally) invertible local transformation
$(\CC{2},0)\longrightarrow(\CC{2},0)$ of the form
\begin{equation}\label{e.gauge}
(z,w) \longrightarrow \left( zf(w),\, g(w) \right) ,
\end{equation}
where $f(w)$ and $g(w)$ are two (formal) power series with
$f(0)\neq 0,\,g(0)=0,\,g'(0)\neq 0$. A \it (formal) special gauge
transformation \rm is a (formally) invertible local transformation
of the form~\eqref{e.gauge}, where $f(w)$ and $g(w)$ are (formal)
power series that satisfy an additional normalization
$f(0)=1,\,g(w)=w+O(w^{m+1})$.
\end{dfn}

Clearly, the set of (formal) gauge transformations, as well as the
set of (formal) special gauge transformations, form a group. We
also note that for a formal gauge transformation the formal
recalculation of derivatives is well-defined (see Section 2), so
that one can correctly define, in the natural way,  formal
equivalence of $m$-admissible linear ODEs by means of gauge
transformations.

\begin{propos} For any $m\geq 2$ and $\beta\in\RR{}$ the ODE $\mathcal
E^m_\beta$ is formally equivalent to the ODE $\mathcal E^m_0$ by
means of a formal special gauge transformation.\end{propos}

\begin{proof} The strategy of the proof is based on finding the fundamental system of formal solutions
of an ODE $\mathcal E^m_\beta$ (we refer to
\cite{ilyashenko},\cite{ai},\cite{vazow},\cite{coddington} for
more information on the concepts of a formal normal form and a
fundamental system of formal solutions). It is straightforward to
verify that the function $\exp\left(\frac{2i}{1-m}w^{1-m}\right)$
is a solution of the ODE $\mathcal E^m_0$, so that the fundamental
system of solutions for $\mathcal E^m_0$ is
$\left\{1,\exp\left(\frac{2i}{1-m}w^{1-m}\right)\right\}$. For the
system $\mathcal E^m_\beta$ with $\beta\neq 0$ we consider the
corresponding system \eqref{mainsystem} and note that the
principal
matrix $A_0=\begin{pmatrix} 0 & 0 \\
0 & 2i\end{pmatrix}$ is diagonal and its eigenvalues are distinct,
hence the system is nonresonant. We first perform a transformation
$\begin{pmatrix} z \\ u
\end{pmatrix}\lr (I+w^{m-1}H)\begin{pmatrix} z \\ u
\end{pmatrix}$, where $I$ is the identity and $H$ is a constant $2\times 2$ matrices,
and obtain the system $\begin{pmatrix} z \\ u
\end{pmatrix}'=\frac{1}{w^m}A(w)\begin{pmatrix} z \\ u
\end{pmatrix}$, where $A(w)$ is a holomorphic matrix-valued
function of the form $A_0+A_{m-1}w^{m-1}+O(w^m)$. Here $A_0$ is
the same as for the initial system, and
$$
A_{m-1} =
\begin{pmatrix} 0 & 1 \\ \beta & 1-m
\end{pmatrix}+A_0H-HA_0 .
$$
By choosing $H=\frac{1}{2i}\begin{pmatrix} 0 & 1 \\ -\beta & 0
\end{pmatrix}$ we may eliminate the nondiagonal elements, and so
$A_{m-1} =\begin{pmatrix} 0 & 0 \\ 0 & 1-m \end{pmatrix}$. We now
follow the Poincare-Dulac formal normalization procedure for
non-Fuchsian systems (see, e.g., \cite{ilyashenko}, Thm 20.7), and
using the fact that the system is nonresonant, bring it to a
polynomial diagonal normal form with the $(m-1)$-jet being equal
to $A_0+w^{m-1}A_{m-1}$. As all terms of the form $O(w^m)$ can be
removed in the nonresonant case, the formal normal form of system
\eqref{mainsystem} becomes
\begin{equation}\label{normalform}
\begin{pmatrix} z \\ u
\end{pmatrix}'=\left[\frac{1}{w^m}\begin{pmatrix} 0 & 0 \\ 0 & 2i
\end{pmatrix}+\frac{1}{w}\begin{pmatrix} 0 & 0 \\ 0 & 1-m
\end{pmatrix}\right]\begin{pmatrix} z \\ u
\end{pmatrix}.\end{equation}
This implies that systems \eqref{mainsystem} for different $\beta$
are formally gauge equivalent, however, our goal is to deduce the
equivalence of the ODEs \eqref{mainODE}, which is a \it different
\rm issue. The normal form \eqref{normalform} admits the
fundamental matrix of solutions
$$e^{\frac{1}{1-m}w^{1-m}\begin{pmatrix} 0 & 0 \\ 0 & 2i
\end{pmatrix}}\cdot w^{\begin{pmatrix} 0 & 0 \\ 0 & 1-m
\end{pmatrix}}.$$ We
conclude from here that the fundamental system of formal solutions
for \eqref{mainsystem} is of the form
\begin{equation} \label{fundamsystem} \hat F_\beta(w)\cdot
e^{\frac{1}{1-m}w^{1-m}\begin{pmatrix} 0 & 0 \\ 0 & 2i
\end{pmatrix}}\cdot w^{\begin{pmatrix} 0 & 0 \\ 0 & 1-m
\end{pmatrix}},\end{equation} where $\hat F_\beta(w)=\begin{pmatrix}f_\beta(w) & g_\beta(w)\\
h_\beta(w) & s_\beta(w)\end{pmatrix}$ is a matrix-valued formal
power series of the form $I+\sum\limits_{k\geq 2}F_kw^k$ ($I$
denotes the unit $2\times 2$ matrix). The latter means that the
columns of \eqref{fundamsystem} are formally linearly independent
and their formal substitution into \eqref{mainsystem} gives the
identity. Representation \eqref{fundamsystem} implies that
equation \eqref{mainODE} possesses a formal fundamental system of
solutions
$\left\{f_\beta(w),g_\beta(w)\cdot\exp\left(\frac{2i}{1-m}w^{1-m}\right)\cdot
w^{1-m}\right\}$ for two formal power series
\begin{equation}\label{formfg}f_\beta(0)=1+O(w),\,g_\beta(w)=w^{m-1}+O(w^m)\end{equation} (the expansion of
$g_\beta$ follows from the fact that, in view of
\eqref{mainsystem},
$$
\left(g_\beta(w)\exp\left(\frac{2i}{1-m}w^{1-m}\right)\right)'=
\frac{1}{w}s_\beta(w)\exp\left(\frac{2i}{1-m}w^{1-m}\right) ,
$$
and also $s_\beta(w)=1+O(w)$, so that $ord_0 g_\beta=m-1$, and
after a scaling we get $g_\beta(w)=w^{m-1}+O(w^m)$).

We set
$$
\chi(w):=\frac{1}{f_\beta(w)},\ \
\tau(w):=w\left(1+\frac{1-m}{2i}w^{m-1}\ln\frac{g_\beta(w)}{w^{m-1}f_\beta(w)}\right)^{\frac{1}{1-m}}.
$$
In view of \eqref{formfg}, $\tau(w)$ is a well defined formal
power series of the form $w+O(w^{m+1})$, and $\chi(w)$ is a well
defined formal power series of the form $1+O(w)$. We claim that
\begin{equation}\label{equivalence}
(z,w) \longrightarrow \left(\chi(w)z, \tau(w) \right)
\end{equation}
is the desired formal special gauge transformation sending
$\mathcal E^m_\beta$ into $\mathcal E^m_0$. This fact can be seen
either from a straightforward computation (one has to perform
substitution \eqref{equivalence} in $\mathcal E^m_\beta$ and use
the fact that
$\left\{f_\beta(w),g_\beta(w)\cdot\exp\left(\frac{2i}{1-m}w^{1-m}\right)\cdot
w^{1-m}\right\}$ is the fundamental system of solutions for
$\mathcal E^m_\beta$), or as follows. As it is shown in
\cite{arnoldgeom}, if two functions $z_1(w),z_2(w)$ are some
linearly independent holomorphic solutions of a second order
linear ODE $z''=p(w)z'+q(w)z$, then the transformation
$z\longrightarrow\frac{1}{z_1(w)}z,\,w\longrightarrow\frac{z_2(w)}{z_1(w)}$
transfers the initial ODE into the simplest ODE $z''=0$. The same
fact can be verified, by a simple computation, for more general
classes of functions, for example, for series of type
$h(w)\cdot\exp\left(a w^{\alpha}\right)$, where $h(w)$ is a formal
Laurent series with a finite principal part, and
$a,\alpha\in\CC{}$ are fixed constants. Then
\begin{equation}\label{straighten}
z\longrightarrow\frac{1}{f_\beta(w)}z,\,
w\longrightarrow\frac{g_\beta(w)}{w^{m-1}f_\beta(w)}\exp\left(\frac{2i}{1-m}w^{1-m}\right)
\end{equation}
transforms formally $\mathcal E^m_\beta$ into $z''=0$, and
\begin{equation}\label{straighten0}
z\longrightarrow z,\,
w\longrightarrow\exp\left(\frac{2i}{1-m}w^{1-m}\right)
\end{equation}
transforms $\mathcal E^m_0$ into $z''=0$. It follows then that the
formal substitution of \eqref{equivalence} into
\eqref{straighten0} gives \eqref{straighten}. Since the chain rule
agrees with the above formal substitutions, this shows that
\eqref{equivalence} transfers $\mathcal E^m_\beta$ into $\mathcal
E^m_0$. This proves the proposition.
\end{proof}

On the other hand, the ODEs $\mathcal E_\beta$ and $\mathcal
E^m_0$ are holomorphically inequivalent for a generic $\beta$, as
the following proposition shows.

\begin{propos}
For any $m\geq 2$ and $\beta\neq l(l-m+1),\,l\in\mathbb{Z}$, the
ODE $\mathcal E^m_\beta$ has a nontrivial monodromy, while the ODE
$\mathcal E^m_0$ has a trivial one.
\end{propos}

\begin{proof} For the ODE $\mathcal
E^m_0$ the fundamental system of holomorphic solutions is given in
$\CC{}\setminus\{0\}$ by
$\left\{1,\exp\left(\frac{2i}{1-m}w^{1-m}\right)\right\}$, so that
all solutions of $\mathcal E^m_0$ are single-valued in
$\CC{}\setminus\{0\}$, accordingly, its monodromy is trivial. One
needs now to obtain the monodromy matrix for a generic system
\eqref{mainsystem}. In order to do that we consider $\infty$ as an
isolated singular point for \eqref{mainsystem} and perform the
change of variables $t:=\frac{1}{w}$. We obtain the system
\begin{equation} \label{systematinfinity}
\begin{pmatrix} y \\ u
\end{pmatrix}'=\left[t^{m-2}\begin{pmatrix} 0 & 0 \\ 0 & -2i
\end{pmatrix}+\frac{1}{t}\begin{pmatrix} 0 & -1 \\ -\beta & m-1
\end{pmatrix}\right]\begin{pmatrix} y \\ u
\end{pmatrix}\end{equation} with an isolated \it Fuchsian \rm singularity
at $t=0$. As \eqref{mainsystem} does not have any more singular
points in $\overline{\CC{}}$ beside $w=0$ and $w=\infty$, it is
sufficient to prove nontriviality of the monodromy matrices at
$t=0$ for systems \eqref{systematinfinity} with $\beta\neq
l(l-m+1),\,l\in\mathbb{Z}$. For the residue matrix
$R_\beta=\begin{pmatrix} 0 & -1 \\ -\beta & m-1
\end{pmatrix}$ of
\eqref{systematinfinity} at $t=0$ denote by $\lambda_1,\lambda_2$
its eigenvalues. The Poincare-Dulac procedure for Fuchsian systems
implies (see, e.g., Corollary 16.20 in \cite{ilyashenko}) that the
collection of eigenvalues of the monodromy operator for
\eqref{systematinfinity} looks as $\{e^{2\pi i\lambda_1},e^{2\pi
i\lambda_2}\}$. In particular, if one of the eigenvalues is not an
integer, the system \eqref{mainsystem} (and the corresponding ODE
$\mathcal E^m_\beta$) has a nontrivial monodromy. Applying the
relations $\lambda_1+\lambda_2=m-1,\,\lambda_1\lambda_2=-\beta$,
we obtain the claim of the proposition.
\end{proof}

Next we need to establish a connection between equivalences of the
$m$-admissible ODEs $\mathcal E^m_\beta$ and the associated real
hypersurfaces. We start with

\begin{propos} The only formal special gauge transformation preserving
the ODE $\mathcal E^m_0$ is the identity. In particular, the only
formal special gauge transformation, transferring $\mathcal
E^m_\beta$ into $\mathcal E^m_0$, is given by \eqref{equivalence}.
\end{propos}

\begin{proof} Let $F:\,z^*=zf(w),\,w^*=g(w),\,f=1+O(w),\,g=w+O(w^{m+1})$
be a formal special gauge self-transformation of $\mathcal E^m_0$.
It is not difficult to calculate that $F^{-1}$ transforms
$\mathcal E^m_0$ into a well-defined formal meromorphic second
order linear ODE
$$\frac{f}{(g')^2}z''+\left(2\frac{f'}{(g')^2}-\frac{fg''}{(g')^3}\right)z'+\left(\frac{f''}{(g')^2}-\frac{f'g''}{(g')^3}\right)z=\left(\frac{2i}{g^m}-\frac{m}{g}\right)\left(\frac{f}{g'}z'+
\frac{f'}{g'}z\right).$$ Comparing the above identity
with~\eqref{mainODE} with $\beta=0$ gives
\begin{gather}
\label{newODE1}\left(\frac{2i}{g^m}-\frac{m}{g}\right)\frac{f'}{g'}-\left(\frac{f''}{(g')^2}-\frac{f'g''}{(g')^3}\right)=0 ,\\
\label{newODE2}g'\left(\frac{2i}{g^m}-\frac{m}{g}\right)-2\frac{f'}{f}+\frac{g''}{g'}=\frac{2i}{w^m}-\frac{m}{w}
.\end{gather}

If  $f\not\equiv 1$, then \eqref{newODE1} gives
$\frac{f''}{f'}=g'\left(\frac{2i}{g^m}-\frac{m}{g}\right)+\frac{g''}{g'}$
and, comparing with \eqref{newODE2}, we obtain
$\frac{f''}{f'}=2\frac{f'}{f}+\frac{2i}{w^m}-\frac{m}{w}$. Making
in the latter for the formal power series $f$ the substitution
$h:=\frac{1}{f}$ (note that $h(w)$ is also a well-defined formal
power series with $h(w)=1+O(w)$) it is not difficult to obtain
that $h''=\left(\frac{2i}{w^m}-\frac{m}{w}\right)h'$, so that $h$
satisfies the initial ODE $\mathcal E^m_0$. But any (formal) power
series solution for $\mathcal E^m_0$ is constant (it can be seen,
for example, from the fact that the fundamental system of
solutions for $\mathcal E^m_0$ is
$\left\{1,\exp\left(\frac{2i}{1-m}w^{1-m}\right)\right\}$), which
contradicts $h\not\equiv 1,f\not\equiv 1$.

Suppose now that $f\equiv 1$. Then \eqref{newODE1} holds
trivially, and we examine \eqref{newODE2}. Assuming that
$g(w)\not\equiv w$,~\eqref{newODE2} can be rewritten as a \it
well-defined \rm differential relation
$$2i\left(\frac{1}{g^{m-1}(1-m)}\right)'-2i\left(\frac{1}{w^{m-1}(1-m)}\right)'+(\ln
g')'-m\left(\ln\frac{g}{w}\right)'=0,$$ which gives
$\frac{2i}{1-m}\left(\frac{1}{g^{m-1}}-\frac{1}{w^{m-1}}\right)+\ln
g'-m\ln\frac{g}{w}=C_1$ for some constant $C_1\in\CC{}$. It
follows then that the formal meromorphic Laurent series
$\frac{1}{1-m}\left(\frac{1}{g^{m-1}}-\frac{1}{w^{m-1}}\right)$ is
in fact a formal power series, and a straightforward computation
shows that the substitution
$\frac{1}{1-m}\left(\frac{1}{g^{m-1}}-\frac{1}{w^{m-1}}\right):=u$,
where $u(w)$ is a formal power series, transforms the latter
equation for $g$ into $2iu+\ln(w^mu'+1)=C_1$. Shifting $u$, we get
the equation $2iu+\ln(w^mu'+1)=0$, where $u(0)=0$. Hence we
finally  obtain the following meromorphic first order ODE for the
formal power series $u(w)$:
\begin{equation}\label{ODEforu}u'=\frac{1}{w^m}(e^{-2iu}-1).\end{equation} However,
\eqref{ODEforu} has no non-zero formal power series solutions. To
see that, we note that for $u\not\equiv 0,\,u(0)=0$
\eqref{ODEforu} can be represented as
$-\frac{1}{2i}u'\left(\frac{1}{u}+H(u)\right)=\frac{1}{w^m}$,
where $H(t)$ is a holomorphic at the origin function. Hence we get
that the logarithmic derivative $\frac{u'}{u}$ has the expansion
$-\frac{2i}{w^m}+...$, where the dotes denote a formal power
series in $w$. But this clearly cannot happen for a formal power
series $u(w)$. Hence $u\equiv 0$, and, returning to the unknown
function $g$, we get $\frac{1}{g^{m-1}}-\frac{1}{w^{m-1}}=C$ for
some constant $C\in\CC{}$, so that
$g(w)=\frac{w}{(1+Cw^{m-1})^\frac{1}{m-1}}$. Taking into account
$g(w)=w+O(w^{m+1})$, we conclude that $C=0$ and $g(w)=w$. This
proves the proposition.
\end{proof}

Let now $\mathcal S=\{w=\rho(z,\bar\xi,\bar\eta)\}$ be a (general)
Segre family in a polydisc
$\Delta_\delta\times\Delta_\varepsilon$. We consider the \it
complex \rm submanifold
\begin{equation}\label{foliated}
\mathcal M_{\mathcal
S}=\left\{(z,w,\xi,\eta)\in\Delta_\delta\times\Delta_\varepsilon\times\Delta_\delta\times\Delta_\varepsilon:
\,w=\rho(z,\xi,\eta)\right\}\subset\CC{4},
\end{equation}
 and call it \it the associated foliated submanifold \rm of the family $\mathcal
 S$. If $\mathcal S$ is associated with an $m$-admissible ODE
 $\mathcal E$, we call $\mathcal M_{\mathcal S}$ \it the associated foliated submanifold of $\mathcal E$. \rm
We call $\mathcal M_{\mathcal S}$ \it $m$-admissible, \rm if
$\mathcal S$ is $m$-admissible. If $S$ is the Segre family of a
real hypersurface $M\subset\CC{2}$, then the associated foliated
submanifold is simply the complexification of $M$. The concept of
the associated foliated submanifold is somewhat analogous to that
of the \it submanifold of solutions \rm of a nonsingular
completely integrable PDE system (see ,e.g.,
\cite{cartanODE},\cite{olver}, \cite{gm},\cite{merker}). Here we
consider the case of \it singular \rm differential equations and
formal mappings between them.

The foliated submanifold $\mathcal M_{\mathcal S}$ admits two
natural foliations. The first one is the initial
foliation~$\mathcal S$ with leaves $\{(z,w,\xi,\eta)\in\mathcal
M_{\mathcal S}:\,\xi=const,\,\eta=const\}$. The second one is the
family of dual Segre varieties with leaves
$\{(z,w,\xi,\eta)\in\mathcal M_{\mathcal S}:\,z=const,w=const\}$.
If now $\mathcal E_1, \mathcal E_2$ are two $m$-admissible ODEs,
then it is crucial for the study of (formal) biholomorphisms
between them to consider (formal) biholomorphisms between the
associated foliated submanifolds $\mathcal M_{\mathcal
S_1},\mathcal M_{\mathcal S_2}$, preserving the origin and both
foliations. Clearly, any such biholomorphism has the form
\begin{equation}\label{twin}
(z,w,\xi,\eta)\longrightarrow(F(z,w),G(\xi,\eta)) ,
\end{equation}
where $F(z,w),G(\xi,\eta)$ are (formal) biholomorphisms
$(\CC{2},0)\longrightarrow(\CC{2},0)$. In this case  we call  the
transformation $(F(z,w),G(\xi,\eta)):\,(\mathcal M_{\mathcal
S_1},0)\longrightarrow(\mathcal M_{\mathcal S_2},0)$ \it a
(formal) coupled transformation of $\mathcal M_{\mathcal S_1}$
into $\mathcal M_{\mathcal S_1}$. \rm

Using the notion of associated foliated submanifolds, one can push
the concept of a Segre family to the formal level. Namely, let
$\rho(z,\xi,\eta)$ be a formal power series without a constant
term and the linear part equal to $\eta$. We then call the formal
complex submanifold $\mathcal M=\{w=\rho(z,\xi,\eta)\}$ of
$\CC{4}$ \it a formal foliated submanifold. \rm A formal foliated
submanifold can be identified with its formal defining function
$\rho$. If, in addition, $\rho$ is as in \eqref{phi}, we call
$\mathcal M$ \it $m$-admissible. \rm If $\mathcal
M=\{w=\rho(z,\xi,\bar\eta)\}$ is a formal foliated submanifold
such that the defining function $\rho(z,\xi,\eta)$ contains $\eta$
as a factor (for example, all $m$-admissible formal foliated
submanifolds have this property), and $\mathcal E$ is an
$m$-admissible ODE, then the derivatives $\rho_z(z,\xi,\bar\eta)$
and $\rho_{zz}(z,\xi,\eta)$ are well-defined power series, and we
say that \it $\mathcal M$ is formally associated with the ODE
$\mathcal E$ \rm if the well-defined substitution of the power
series $\rho(z,\xi,\eta)$ into the inverse ODE to $\mathcal E$
gives the identity of the formal power series in $z,\xi,\eta$ on
both sides of the equation.

Let now $\mathcal E_1,\mathcal E_2$ be two $m$-admissible ODEs,
$\mathcal M_1$ be a foliated submanifold, associated with
$\mathcal E_1$, and $F(z,w):\,(\CC{2},0)\lr(\CC{2},0)$ be a formal
invertible mapping tangent to the identity map at the origin. Then
the formal recalculation of the derivatives
$z_w,z_{ww},w_z,w_{zz}$ is well-defined (see Section 2), and one
can correctly define the formal equivalence of $\mathcal
E_1,\mathcal E_2$ by means of $F$. In addition, consider a similar
formal transformation $G(\xi,\eta):\,(\CC{2},0)\lr(\CC{2},0)$ of
the space of parameters $\xi,\eta$. One can then correctly define
the image of the foliated submanifold $\mathcal M_1$ under the
formal direct product
$(F(z,w),G(\xi,\eta)):\,(\CC{4},0)\lr(\CC{4},0)$ and obtain a
unique formal foliated submanifold $\mathcal M$ (one needs to
substitute $(F^{-1},G^{-1})$ into $\mathcal M_1$ and apply the
implicit function theorem in the category of formal power series).
It is immediate then that for any formal invertible transformation
$F(z,w)$, transferring $\mathcal E_1$ into $\mathcal E_2$, and any
formal invertible transformation $G(\xi,\eta)$ in the space of
parameters, where both $F$ and $G$ are tangent to the identity at
zero, the image of $\mathcal M_1$ under the direct product
$(F(z,w),G(\xi,\eta))$ is a foliated submanifold $\mathcal M_2$,
associated with $\mathcal E_2$.

Consider then a (formal) special gauge transformation
$(z,w)\longrightarrow F(z,w)=(zf(w),g(w))$, transforming an
$m$-admissible ODE $\mathcal E_1$ into an $m$-admissible ODE
$\mathcal E_2$. Let $\mathcal S_1,\mathcal S_2$ be the associated
positive $m$-admissible Segre families and $\mathcal M_1,\mathcal
M_2$ the associated foliated submanifolds. We claim that there
exists a (formal) special gauge transformation
$(\xi,\eta)\longrightarrow G(\xi,\eta)=(\xi
\lambda(\eta),\mu(\eta))$, such that
$(z,w,\xi,\eta)\longrightarrow(F(z,w),G(\xi,\eta))$ is a (formal)
coupled  transformation of $\mathcal M_1$ into $\mathcal M_2$.
Indeed, let us first prove

\begin{lem} There exists a unique (formal) special gauge transformation
$$(\xi,\eta)\longrightarrow G(\xi,\eta)=(\xi \lambda(\eta),\mu(\eta))$$
such that the (formal) transformation
$(z,w,\xi,\eta)\longrightarrow(F(z,w),G(\xi,\eta))$ sends
$\mathcal M_1$ into an $m$-admissible (formal) foliated
submanifold $\mathcal M$.
\end{lem}

\begin{proof}
To simplify notations we will prove the same statement for the
special gauge mapping $F^{-1}$ of the ODE $\mathcal E_2$. Let
$\mathcal M_2$ be given by \eqref{foliated} with $\psi$ as in
\eqref{phi}. Our goal is to determine uniquely two (formal) power
series $\lambda(\eta),\mu(\eta)$ with
$\lambda(\eta)=1+O(\eta),\,\mu(\eta)=\eta+O(\eta^{m+1})$ such that
\begin{equation}\label{twingauge}
g(w)=\mu(\eta)e^{i\mu(\eta)^{m-1}\psi(zf(w),\xi\lambda(\eta))}\end{equation}
defines an $m$-admissible foliated submanifold. Note that
\eqref{twingauge} can be represented as
\begin{equation}\label{twingauge1}
g(w)=\mu(\eta)+i\bar\mu(\eta)^mz\xi
f(w)\lambda(\eta)+O(z^2\xi^2\eta^m),
\end{equation}
from which we conclude that \eqref{twingauge} defines a formal
foliated submanifold of the form $w=\sum\limits_{j\geq 0}
\varphi_j(\eta)(z\xi)^j$ with $\varphi_0(\eta)=O(\eta)$ and
$\varphi_j(\eta)=O(\eta^m)$ for $j\geq 1$. Hence we are interested
in the choice of $\lambda(\eta),\mu(\eta)$ which gives
$\varphi_0(\eta)=\eta,\,\varphi_1(\eta)=i\eta^m$. The latter is
equivalent to the fact that the substitution $w=\eta+i\eta^m
z\xi+O(z^2\xi^2\eta^m)$ (corresponding to the desired target
foliated submanifold $\mathcal M$) into \eqref{twingauge1} makes
\eqref{twingauge1} an identity modulo $z^2\bar\xi^2$. Thus we get
$g(\eta)+i\eta^m g'(\eta)z\xi=\mu(\eta)+i\mu(\eta)^mz\xi
f(\eta)\lambda(\eta)+O(z^2\xi^2)$, which is equivalent to
\begin{equation}\label{twingauge2}g(\eta)=\mu(\eta),\,\eta^m
g'(\eta)=\mu(\eta)^m f(\eta)\lambda(\eta).
\end{equation}
Equations \eqref{twingauge2} enable one to determine
$\lambda(\eta),\mu(\eta)$ with the desired properties uniquely,
and this proves the lemma.
\end{proof}

If now $G(\xi,\eta)$ is the special gauge transformation, provided
by Lemma~4.5, it follows from the above arguments that the
(formal) image of $\mathcal M_1$ under the direct product
$(F(z,w),G(\xi,\eta))$ is a (formal) \it $m$-admissible \rm
foliated submanifold $\mathcal M$, associated with $\mathcal E_2$.
However, it is not difficult, in the same manner as in the proof
of Proposition 3.5, to show that even in the formal category the
associated $m$-admissible foliated submanifold is unique (since
the uniqueness follows from the uniqueness of the solution of the
Cauchy problem for the holomorphic ODE \eqref{findphi2}, which
holds true in the formal category as well, see \cite{ck}). Thus we
conclude $\mathcal M=\mathcal M_{2}$, and this proves the
existence of the special gauge transformation $G$ in both
holomorphic and formal settings.

Conversely, let
$(z,w,\xi,\eta)\longrightarrow(F(z,w),G(\xi,\eta))$ be a (formal)
coupled transformation, sending $\mathcal M_{1}$ into $\mathcal
M_{2}$, where both $F$ and $G$ are special gauge. It is easy to
check, by a computation similar to those in Proposition~4.4, that
$F(z,w)$ transfers $\mathcal E_1$ into some (formal)
$m$-admissible ODE $\mathcal E$. On the other hand,
$(F(z,w),G(\xi,\eta))$ (formally) transfers $\mathcal M_{1}$ into
$\mathcal M_{2}$, so that $\mathcal M_{2}$ is (formally)
associated with $\mathcal E$. This shows that $\mathcal E=\mathcal
E_2$ in the case of a holomorphic coupled  transformation. To
treat the formal case we note that relations
\eqref{relation22},\eqref{relation33} similarly hold for formal
$m$-admissible families, associated with formal $m$-admissible
ODEs (the proof does not change), so that the conclusion $\mathcal
E=\mathcal E_2$ holds true in the formal case as well.

We summarize the above arguments in the following

\begin{propos}
Let $\mathcal E_1,\mathcal E_2$ be two $m$-admissible ODEs, and
$\mathcal M_1,\mathcal M_2\subset\CC{4}$ the associated foliated
submanifolds. There is a one-to-one correspondence
$F(z,w)\longrightarrow(F(z,w),G(\xi,\eta))$ between (formal)
special gauge equivalences $F(z,w)$, transforming $\mathcal E_1$
into $\mathcal E_2$, and (formal) coupled transformations
$(F(z,w),G(\xi,\eta))$, sending $\mathcal M_1$ into $\mathcal
M_2$.
\end{propos}

We are now ready to prove the main result of this section. It is a
more detailed version of Theorem~A.

\begin{thm}
For any $m\geq 2$ and $\beta\neq l(l-m+1),\,l\in\mathbb{Z}$,  the
nonminimal at the origin real hypersurface
$M^m_\beta\subset\CC{2}$, associated with the ODE $\mathcal
E^m_\beta$ as in \eqref{mainODE}, is formally equivalent at the
origin to the hypersurface $M^m_0$ by means of the formal special
gauge transformation \eqref{equivalence}, but is locally
biholomorphically inequivalent to $M^m_0$.\end{thm}

\begin{proof}
Consider the foliated submanifolds $\mathcal M^m_\beta$,
associated with $\mathcal E^m_\beta$. It follows from the
definitions of the associated real submanifold and the associated
foliated submanifold that $\mathcal M^m_\beta$ is the
complexification of $ M^m_\beta$. Considering now the reality
condition \eqref{realty} for $ M^m_\beta$ and complexifying it, we
conclude that $\mathcal M^m_\beta$ is invariant under the
anti-holomorphic linear mapping
$\sigma:\,\CC{4}\longrightarrow\CC{4}$ given by
\begin{equation}\label{sigma}(z,w,\xi,\eta)\longrightarrow(\bar\xi,\bar\eta,\bar z,\bar
w).\end{equation} Let now $F(z,w)$ be the formal special gauge
equivalence, provided by Proposition~4.2, and
$(F(z,w),G(\xi,\eta))$ the formal coupled  special gauge
transformation between $\mathcal M^m_\beta$ and $\mathcal M^m_0$,
provided by Proposition~4.6. Then we get that
$\sigma\circ(F(z,w),G(\xi,\eta))\circ\sigma=(\bar G(z,w),\bar
F(\xi,\eta)$ is also a formal coupled  special gauge
transformation between $\mathcal M^m_\beta$ and $\mathcal M^m_0$.
Applying now Proposition~4.4, we conclude that $G(\xi,\eta)=\bar
F(\xi,\eta)$. The latter fact immediately implies that the
transformation $(F(z,w),G(\xi,\eta))$ is the complexification of
$F(z,w)$ (see Section 2), so that $F(z,w)$ maps $M^m_\beta$ into
$M^m_0$ formally.

To prove finally the nonequivalence of  $M^m_\beta$ and $ M^m_0$
for $\beta\neq l(l-m+1),\,l\in\mathbb{Z}$, we use the fact that
each $M^m_\beta$ is Levi nondegenerate and spherical in
$M^m_\beta\setminus X$, where $X=\{w=0\}$ is the complex locus. As
it was explained in the proof of Theorem 3.15, for a fixed point
$p=(z_0,w_0)\in M^m_\beta\setminus X$ and two fixed solutions
$h_1(w),h_2(w)$ of $\mathcal E^m_\beta$ near $p$ with
$h_1(w_0)\neq 0$, one of the possible mappings $\Lambda$ of
$M^m_\beta$ into a quadric $\mathcal Q\subset\CP{2}$ is given by
\eqref{maptosphere}. Clearly, $\Lambda$ has a trivial monodromy
about the complex locus $X$ if and only if both $h_1(w),h_2(w)$
have a trivial monodromy about the origin, and the latter is
equivalent to the fact that the ODE $\mathcal E^m_\beta$ has a
trivial monodromy at $w=0$. Now the desired statement follows from
Proposition~4.3 and the fact that the monodromy of a mapping into
a quadric for a nonminimal hypersurface, Levi nondegenerate and
spherical outside the complex locus, is a biholomorphic invariant
(see \cite{nonminimal}). This completely proves the theorem.
\end{proof}

\smallskip

\noindent\bf Proof of statement (a) of Theorem C. \rm The main
step of the proof is the generalization of the constructions of
Theorem 4.7 to hypersurfaces in $\CC{N}$ with  $N\geq 3$. Fix
$m\geq 2$ and $\beta\neq l(l-m+1),\,l\in\mathbb{Z}$, and suppose
that $M^m_\beta,\ M^m_0 \subset \mathbb C^2$ are given near the
origin by the defining equations $\im w = \theta(z\bar z,\re w)$
and $\im w = \theta'(z\bar z,\re w)$. We also denote the mapping
\eqref{equivalence} by $F(z,w)=(zf(w),g(w))$ and the coordinates
in $\CC{N}$ by $z_1,...,z_{N-1},w$. Then it is not difficult to
see that the formal invertible mapping $H:\,(z_1,...,z_{N-1},w)\lr
(z_1f(w),...,z_{N-1}f(w),g(w))$ transfers the smooth real-analytic
nonminimal at the origin hypersurface $M=\{\im w=\theta(z_1\bar
z_1+...+z_{N-1}\bar z_{N-1},\re w)\}$ formally into the smooth
real-analytic nonminimal at the origin hypersurface $M'=\{\im
w=\theta'(z_1\bar z_1+...+z_{N-1}\bar z_{N-1},\re w)\}$. Since
$M^m_\beta$ and $M^m_0$ are Levi nondegenerate outside the complex
locus $\{w=0\}$, the same holds true for $M$ and $M'$, so that $M$
and $M'$ are holomorphically nondegenerate.

It can be seen from the proof of Theorem 4.7 that for any choice
of a single-valued branch of the mapping $\Lambda$, the target
quadric $\mathcal Q$, considered in the affine chart
$\CC{2}\subset\CP{2}$, is invariant under the rotations $z^*\lr
e^{it}z^*,t\in\RR{}$. Thus one can argue as in the proof of
Theorem 4.7 and consider, in the spirit of \eqref{maptosphere},
the mapping
$$\Lambda_n:\,(z_1,...,z_{N-1},w)\longrightarrow\left(\frac{z_1}{h_1(w)},...,\frac{z_{N-1}}{h_1(w)},\frac{h_2(w)}{h_1(w)}\right),$$
where $h_1(w)$ and $h_2(w)$ are some linearly independent analytic
solutions of the ODE $\mathcal E^m_\beta$ in $\CC{}\setminus
\{0\}$. Since $\Lambda$ sends a germ of $M^m_\beta$ at a Levi
nondegenerate point into a quadric $\mathcal Q\subset\CP{2}$, the
mapping $\Lambda_n$ transfers a germ of $M$ at a Levi
nondegenerate point into a \it nondegenerate \rm quadric $\mathcal
Q_N\subset\CP{N}$, obtained from $\mathcal Q$ by the substitution
of $z_1\bar z_1+...+z_{N-1}\bar z_{N-1}$ for $z\bar z$. Since
$\Lambda$ has a nontrivial monodromy,  we conclude that the
nonminimal hypersurface $M$ has a nontrivial monodromy operator in
the sense of \cite{nonminimal}. In a similar way we deduce that
the monodromy operator of the nonminimal hypersurface $M'$ is
trivial. Hence, $M$ and $M'$ are holomorphically inequivalent at
the origin. This proves the theorem in the hypersurface case.

For each class of CR-submanifolds of codimension $k\geq 2$ and
CR-dimension $n\geq 1$ we consider the holomorphically
nondegenerate CR-submanifolds $P=M\times \Pi_{k-1}$ and
$P'=M'\times \Pi_{k-1}$, where $M,M'\subset\CC{n+1}$ are chosen
from the hypersurface case and $\Pi_{k-1}\subset\CC{k-1}$ is the
totally real plane $\im W=0,W\in\CC{k-1}$. Then the direct product
of the above mapping $H$ and the identity map gives a divergent
formal equivalence between $P$ and $P'$. Finally, to show  that
$P$ and $P'$ are inequivalent holomorphically, we  denote the
coordinates in $\CC{n+k}$ by $(Z,W),\,Z\in\CC{n+1},W\in\CC{k-1}$
and note that, since $\Pi$ is totally real, for each holomorphic
equivalence
$$
(\Phi(Z,W),\Psi(Z,W)):\,(M\times \Pi_{k-1},0)\lr (M'\times
\Pi_{k-1},0) ,
$$
one has $\Psi(Z,W)=\Psi(W)$ for a vector power series $\Psi(W)$
with real coefficients and $\Psi(0)=0$. Since the initial mapping
$(\Phi(Z,W),\Psi(Z,W))$ is invertible at $0$, we conclude that the
mapping $\Phi(Z,0):\,(\CC{n},0)\lr(\CC{n},0)$ is invertible at $0$
as well, and since $(\Phi(Z,W),\Psi(W)):\,(M\times \Pi_{k-1},0)\lr
(M'\times \Pi_{k-1},0)$, the map $\Phi(Z,0)$ is a local
equivalence between $(M,0)$ and $(M',0)$. Now the desired
statement is obtained from the hypersurface case. This proves
statement (a) of the theorem.

\section{Real hypersurfaces with divergent CR-automorphisms}

As an application of the Theorem~4.7 we will show in this section
that a generic hypersurface $M^m_\beta$ from Section 4 with $m\geq
2$ has the following property: there exists a \it divergent \rm
formal vector field of the form $L = zA(w)\dz+B(w)\dw$, vanishing
to order $m$ at zero, such that its real part $\re L = L+\bar L$
is formally tangent to $M^m_\beta$. In particular, the formal flow
of $\re L$ provides generically \it divergent \rm formal
automorphisms of $(M^m_\beta,0)$.

We start with a detailed study of the real hypersurfaces
$M^m_0\subset\CC{2}$. It turns out that these hypersurfaces can be
described explicitly, namely, using elementary functions. Fix an
integer $m\geq 2$ and recall that the fundamental system of
holomorphic solutions for the ODE $\mathcal E^m_0$ is given in
$\CC{}\setminus\{0\}$ by
$\left\{1,\exp\left(\frac{2i}{1-m}w^{1-m}\right)\right\}$.
Applying \eqref{maptosphere}, we obtain that the locally
biholomorphic map
\begin{equation}\label{maptosphere0}
\Lambda: (Z,W) = \left(z,\, e^{\frac{2i}{1-m}w^{1-m}} \right)
\end{equation}
maps $\mathcal E^m_0$ into the simplest equation $Z_{WW}=0$.
Consider now the real hyperquadric
$$\mathcal Q=\left\{2|Z|^2+|W|^2=1 \right\}\subset\CC{2},$$
linearly equivalent to the standard sphere $S^3\subset\CC{2}$. We
claim that $\Lambda^{-1}(\mathcal Q)$ contains the Levi
nondegenerate part of the desired hypersurface $M^m_0$. Indeed,
the set $\Lambda^{-1}(\mathcal Q)\subset\CC{2}$ can be described
as
$$2|z|^2+e^{\frac{2i}{1-m}w^{1-m}}\cdot e^{\frac{-2i}{1-m}\bar
w^{1-m}}=1,$$ so that it contains the set
$\frac{2i}{1-m}w^{1-m}=\frac{2i}{1-m}\bar
w^{1-m}+\ln(1-2|z|^2),\,|z|^2<\frac{1}{2}$, and the union of this
real-analytic set, considered in a sufficiently small polydisc
$U\ni 0$, with the complex line $\{w=0\}$ contains the component
\begin{equation}\label{Mm0}
w=\bar w\left(1+\frac{i}{2}(1-m)\bar
w^{m-1}\ln\frac{1}{1-2|z|^2}\right)^{\frac{1}{1-m}} .
\end{equation}
Since $\Lambda$ is locally biholomorphic in
$\CC{2}\setminus\{w=0\}$, equation \eqref{Mm0} defines in the
polydisc $U\ni 0$ a smooth real-analytic nonminimal at the origin
real hypersurface $M$. As the right hand side of \eqref{Mm0} has
the expansion $\bar w+i\bar w^m |z|^2+O(\bar w^m|z|^4)$, we
conclude that $M$ is $m$-admissible. The mapping $\Lambda$ maps
locally biholomorphically each of the two sides $\{\re w>0\}$ and
$\{\re w<0\}$ of $M$ into $\mathcal Q$. Since all Segre varieties
$Q_{(A,B)}$ of $\mathcal Q$ with $A\neq 0$ satisfy the simplest
ODE $Z_{WW}=0$, and $\Lambda$ transforms the ODE $\mathcal E^m_0$
into $Z_{WW}=0$, we conclude that all Segre varieties $Q_{(a,b)}$
of $M$ with $a,b\neq 0$ satisfy the ODE $\mathcal E^m_0$. Hence
$M$ is an $m$-admissible real hypersurface, associated with
$\mathcal E^m_0$, and we finally conclude from Proposition 3.5
that $M=M^m_0$, so that the hypersurfaces $M^m_0$ are given by
\eqref{Mm0} for each $m\geq 2$.

Consider now a holomorphic vector field
$X=2iW\frac{\partial}{\partial W}\in\mathfrak{hol}(\mathcal Q)$.
Computation shows
 that its pull-back under the mapping
$\Lambda$ near each point with $w\neq 0$ equals $w^m\dw$. This
holomorphic vector field extends to the origin holomorphically,
and we conclude that
$$
L^m_0=w^m\dw\in\mathfrak{hol}(M^m_0,0).
$$
We may construct the desired divergent formal vector field,
tangent to a hypersurface $M^m_\beta$ with $m\geq 2$ and
$\beta\neq l(l-m+1),\,l\in\mathbb{Z}$, by pulling back the vector
field $L^m_0$ with the invertible formal mapping
\eqref{equivalence} (we denote it by $\Phi$ in what follows).
Since the real flow $F^t$ of the vector field $L^m_0$ preserves
$(M^m_0,0)$, and $\Phi$ formally transforms $(M^m_\beta,0)$ into
$(M^m_0,0)$, the well-defined real flow $H^t:=\Lambda\circ F^t
\circ \Lambda^{-1}$ preserves formally $(M^m_\beta,0)$, and the
derivation of $H^t$ at $t=0$ gives a formal vector field
$L^m_\beta$ such that its real part is formally tangent to
$(M^m_\beta,0)$. As follows from the construction, $L^m_\beta$ can
be obtained from $L^m_0$ using the usual chain rule. Since $L^m_0$
vanishes to order $m$, we conclude that the same holds for
$L^m_\beta$. Using the facts that
$\Phi(z,w)=(z\chi(w),\tau(w)),\,\chi(w)=1+O(w),\tau(w)=w+O(w^{m+1})$
(see Proposition 4.2), we finally calculate
\begin{equation}\label{Lmbeta}
L^m_\beta=-\frac{\chi'\tau^m}{\chi
\tau'}z\dz+\frac{\tau^m}{\tau'}\dw=A(w)z\dz+B(w)\dw.
\end{equation}

Below we formulate the main result of this section, which is a
detailed formulation of  Theorem~B.

 \begin{thm}
For any $m\geq 2$ and $\beta\neq l(l-m+1),\,l\in\mathbb{Z}$, the
germ $(M^m_\beta,0)$ admits a divergent formal infinitesimal
automorphism $L^m_\beta$ given by \eqref{Lmbeta} and vanishing to
order $m$. The maps $\chi$ and $\tau$ defined
by~\eqref{equivalence}. The real formal flow $F_t(z,w)$, generated
by $L^m_\beta$, consists of divergent formal automorphisms of
$(M^m_\beta,0)$ for all $t\in\RR{}\setminus C$, where $C$ is a
cyclic subgroup in $(\RR{},+)$.
 \end{thm}

 \begin{proof}
The proof is based on the detailed analysis of the proof of
Proposition~4.2. First, we show that the formal power series
$B(w)$ in \eqref{Lmbeta} is divergent. We denote by $\mathbb
C[[w]]$ the algebra of formal power series in $w$ and by
$\Upsilon$ the linear space of formal series of the form
$f(w)w^{-m}\exp\left(\frac{2i}{1-m}w^{1-m}\right)$, where $f(w)\in
\mathbb C[[w]]$. Recall that $z_1(w)=f_\beta(w)\in \mathbb C[[w]]$
and $z_2(w)=g_\beta(w)\cdot
w^{1-m}\cdot\exp\left(\frac{2i}{1-m}w^{1-m}\right)\in \Upsilon$
form the fundamental system of formal solutions for $\mathcal
E^m_\beta$. It is not difficult to verify, by combining the facts
that $z_1(w)$ and $z_2(w)$ satisfy the ODE $\mathcal E^m_\beta$,
that for the well-defined formal Wronskian
$D(w)=z_2'z_1-z_1'z_2\in \Upsilon$ the classical
Liouville-Ostrogradsky formula holds:
\begin{equation} \label{lo}
D'(w)=\left(\frac{2i}{w^m}-\frac{m}{w}\right)D(w) .
\end{equation}
Since $D(w)\in \Upsilon$, we obtain from \eqref{lo} that
$D(w)=C_0w^{-m}\exp\left(\frac{2i}{1-m}w^{1-m}\right),\,C_0\in\CC{}$,
so that the element $D(w)\in \Upsilon$ is convergent. We claim
that the ratio $\frac{g_\beta(w)}{w^{m-1}f_\beta(w)}\in \mathbb
C[[w]]$ is divergent. For otherwise, we conclude that
$\frac{z_2(w)}{z_1(w)}\in \Upsilon$ is convergent as well, and get
from the relation
$(z_1(w))^2\left(\frac{z_2(w)}{z_1(w)}\right)'=D(w)$ that $z_1(w)$
is convergent, and hence that the mapping \eqref{equivalence} is
convergent, which contradicts Proposition 4.3. Now, from the
definition of $\tau(w)$, we conclude that $\tau(w)$ is divergent,
and \eqref{Lmbeta} shows that $B(w)=(1-m)/(\tau^{1-m})'$ is
divergent, which proves the divergence of the vector field
$L^m_\beta$.

Finally, to prove the divergence of a generic transformation in
the flow of $L^m_\beta$ we consider the one-dimensional divergent
formal vector field $Y=B(w)\dw$, vanishing to order $m$. We then
apply to $Y$ the theory of Ecalle-Voronin (we refer to
\cite{ilyashenko} for details). Denote by $H^t(w)$ the formal flow
of $Y$, and assume that it contains a convergent transformation
$H^{t_0}(w),\,t_0\neq 0$. In the terminology of \cite{ilyashenko},
the convergent transformations in $H^t(w)$ with $t\neq 0$ are \it
parabolic germs, \rm and, as the vector field $Y$ is divergent,
$H^{t_0}(w)$ is a nonembeddable parabolic germ (its Ecalle-Voronin
invariants are nontrivial). As any convergent transformation in
$H^t(w)$ commutes with $H^{t_0}(w)$, it necessarily lies in the
centralizer of $H^{t_0}(w)$, and it follows from the
Ecalle-Voronin theory that the set $\{t\in\CC{}:\, H^t(w)\,\,
\mbox{is convergent}\}$ is contained in a cyclic subgroup of
$(\RR{},0)$, generated by some $c\in\RR{}$. Now the desired
divergence statement follows from the simple relation
\eqref{Lmbeta} between $L^m_\beta$ and $Y$. The theorem is
completely proved now.
\end{proof}

\begin{proof}[Proof of statement (b) of Theorem C.]
The arguments of the proof are similar to those of the  proof of
statement (a) in Theorem C (see Section 4). We fix $m\geq 2$,
$\beta\neq l(l-m+1),\,l\in\mathbb{Z}$, and $N\geq 3$. Arguing
identically to the proof of statement (a), we construct, using the
real hypersurface $M^m_\beta,$ a smooth real-analytic nonminimal
at the origin holomorphically nondegenerate hypersurface
$M\subset\CC{N}$. Then it is not difficult to see from the fact
that the real part of the divergent formal vector field
$L^m_\beta=A(w)z\dz+B(w)\dw$ is formally tangent to $M^m_\beta$
that the real part of the divergent formal vector field
$L=A(w)\left(z_1\frac{\partial}{\partial
z_1}+...+z_{N-1}\frac{\partial}{\partial z_{N-1}}\right)+B(w)\dw$
is formally tangent to $M$. The vector field $L$ vanishes to order
$m$. The divergence statement for the elements of the real flow of
$L$ can be verified in the same way as in the proof of Theorem
5.1. This completely proves Theorem~C.
\end{proof}

Note that Corollary~1.1 follows directly from Theorem C.

\begin{rema} As can be verified, for example, from \cite{yel}, solutions of
the ODEs $\mathcal E^m_\beta$ with arbitrary $\beta\in\RR{}$ can
be described using the Bessel functions. Accordingly, it is
possible to follow the above method and describe the real
hypersurfaces $M^m_\beta$ in terms of Bessel functions. However,
the required computations are quite involved and we do not provide
them here.\end{rema}

\medskip

In conclusion we would like to formulate some of open questions.
The first one concerns the holomorphic and formal isotropy
dimensions (see the Introduction) for a Levi nonflat hypersurface
$M\subset\CC{2}$. The investigation of these two characteristics
of a real hypersurface goes back to Poincare~\cite{poincare}, who
proved the bound $\mbox{dim}\,\mathfrak{aut}\,(M,0)\leq 5$ for the
holomorphic isotropy dimension of a Levi nondegenerate
hypersurface. Combining the known results in the holomorphic
category with the convergence results in \cite{ber1},\cite{jl2},
one can deduce the bounds $\mbox{dim}\,\mathfrak{aut}\,(M,0)\leq
5,\,\mbox{dim}\,\mathfrak{f}\,(M,0)\leq 5$ for all  minimal
hypersurfaces, \rm as well as for 1-nonminimal ones. In the
upcoming paper~\cite{KS-new} the authors prove the bound
$\mbox{dim}\,\mathfrak{aut}\,(M,0)\leq 5$ for an  arbitrary \rm
Levi nonflat hypersurface. Somewhat surprisingly, for the formal
isotropy dimension even its finiteness does not seem to follow
from any known results. As Theorem B shows, the formal and
holomorphic dimensions do not coincide in general, so that the
bound $\mbox{dim}\,\mathfrak{f}\,(M,0)\leq 5$ can not be verified
from the holomorphic case. This leads to the following

\medskip

\noindent{\bf Conjecture 5.3.} The bound
$\mbox{dim}\,\mathfrak{f}\,(M,0)\leq 5$ holds  for an arbitrary
real-analytic Levi nonflat germ $(M,0)\subset\CC{2}$, in
particular, $\mbox{dim}\,\mathfrak{f}\,(M,0)< \infty$.

\medskip

The above question becomes even more delicate if one considers the
isotropy group $\mbox{Aut}\,(M,0)$ as well as the formal isotropy
group $\mathcal F(M,0)$. The group structure results in
\cite{jl1},\cite{jl2} were obtained in the settings where {\it a
posteriori} $\mbox{Aut}\,(M,0)= \mathcal F(M,0)$ and
$\mathfrak{aut}\,(M,0)=\mathfrak{f}\,(M,0)$. Since the
$m$-nonminimal case with $m\geq 2$ is significantly different in
the sense that  $\mbox{Aut}\,(M,0)\subsetneq \mathcal F(M,0)$ and
$\mathfrak{aut}\,(M,0)\subsetneq\mathfrak{f}\,(M,0)$ in general,
 it is interesting to establish a
connection between the objects $\mathfrak{aut}\,(M,0)$,
$\mathfrak{f}\,(M,0)$, $\mbox{Aut}\,(M,0)$ and $\mathcal F(M,0)$,
as well as the group structures for $\mbox{Aut}\,(M,0)$ and
$\mathcal F(M,0)$ in the case $m\geq 2$.


\end{document}